\documentclass[12pt]{amsart}

\usepackage{amssymb}

\usepackage{enumitem}

\usepackage{graphicx}
\usepackage{xcolor}
\usepackage{comment}
\makeatletter
\@namedef{subjclassname@2010}{%
  \textup{2020} Mathematics Subject Classification}
\makeatother

\usepackage[T1]{fontenc}
\usepackage{amsmath,amssymb,amsthm,amsfonts,latexsym,amscd,hyperref}
\allowdisplaybreaks
\hypersetup{colorlinks,linkcolor={blue},citecolor={blue},urlcolor={blue}}
\usepackage{xcolor}
\usepackage{mathtools}
\newtheorem{thm}{Theorem}[section]

\newtheorem{lem}[thm]{Lemma}
\newtheorem{prop}[thm]{Proposition}
\theoremstyle{definition}
\newtheorem{defn}[thm]{Definition}
\theoremstyle{remark}
\newtheorem{rem}[thm]{Remark}
\numberwithin{equation}{section}
\newcommand{\vertiii}[1]{{\left\vert\kern-0.25ex\left\vert\kern-0.25ex\left\vert #1 
    \right\vert\kern-0.25ex\right\vert\kern-0.25ex\right\vert}}

\begin{document}

%%%%% To ease editing, for IMPAN journals add:

\baselineskip=17pt

%%%%%%%%%%%%%%%%

\title[]{Weighted Bilinear Multiplier Theorems in Dunkl Setting via Singular Integrals}
\author[SUMAN MUKHERJEE]{Suman Mukherjee$^\dagger$}
\address{Suman Mukherjee, School of Mathematical Sciences, National Institute of Science Education and Research\\
Bhubaneswar 752050, India. \newline Homi Bhabha National Institute, Training School Complex, Anushaktinagar, Mumbai 400094, India.}
\email{sumanmukherjee822@gmail.com} 
\author[Sanjay Parui]{Sanjay Parui}
\address{Sanjay Parui, School of Mathematical Sciences, National Institute of Science Education and Research\\
Bhubaneswar 752050, India. \newline Homi Bhabha National Institute, Training School Complex, Anushaktinagar, Mumbai 400094, India.}
\email{parui@niser.ac.in}
\thanks{First author is supported by research fellowship  from Department of Atomic Energy (DAE), Government of India.}
\thanks{$\dagger$ \tt{Corresponding author}}
\date{}

\begin{abstract}
The purpose of this article is to present one and two-weight inequalities for bilinear multiplier operators in Dunkl setting with multiple Muckenhoupt weights. In order to do so, new results regarding Littlewood-Paley type theorems and weighted inequalities for multilinear Calder\'on-Zygmund operators in Dunkl setting are also proved.
\end{abstract}


\maketitle
\pagestyle{myheadings}
\markboth{SUMAN MUKHERJEE AND SANJAY PARUI}{BILINEAR MULTIPLIER THEOREMS IN DUNKL SETTING ...}

% The beginning of the text.

\section{Introduction}\label{intro}
 In 1989 C.F. Dunkl \cite{Dunkl} introduced Dunkl operators which can be treated as generalizations of the ordinary directional derivatives. It turns out that they are very useful tools to  study special functions with reflection symmetries. In fact, those operators allow one to build  up a parallel theory to that of classical Fourier analysis, which inherits most of the similar results. This Fourier analytic aspect of Dunkl theory has been intensively studied by several mathematicians over the past three decades. 
\par One of the well studied and trending topics in modern Harmonic analysis is the Fourier multipliers and its multilinear versions. For $f_1,\,f_2\in \mathcal{S}(\mathbb{R}^d)$, the class of all Schwartz functions, the bilinear Fourier multiplier operators is defined as 
\begin{equation*}
{\bf T}_{\bf m}(f_1,f_2)(x)=\int_{\mathbb{R}^{2d}}{\bf m}(\xi,\eta)\mathcal{F}f_1(\xi)\mathcal{F}f_2(\eta)e^{ix\cdot (\xi+\eta)}\,d\xi\,d\eta,
\end{equation*}
where ${\bf m}$ is some reasonable function on $\mathbb{R}^{2d}$ and $\mathcal{F}$ is the classical Fourier transform on $\mathbb{R}^d.$ The classical Coifman-Meyer \cite{CoifmenNHAOTAPIBL} (bilinear) multiplier theorem from 1970's states that if ${\bf m}$ is a bounded function on $\mathbb{R}^{2d}$, which is smooth away from the origin and satisfies the decay condition:
\begin{equation}\label{coif-myer condition classical setting}
|\partial^{\alpha}_{\xi}\partial^{\beta}_{\eta}{\bf m}(\xi,\eta)|\leq C_{\alpha,\,\beta}\left(|\xi|+|\eta|\right)^{-(|\alpha|+|\beta|)}
\end{equation}
for all multi-indices $\alpha,\beta \in \left(\mathbb{N}\cup\{0\}\right)^d$, then the operator ${\bf T}_{\bf m}$ is bounded from $L^{p_1}(\mathbb{R}^d)\times L^{p_2}(\mathbb{R}^d)$ to $L^p(\mathbb{R}^d)$ for $1<p,p_1,p_2<\infty$ with the relation $1/p=1/p_1+1/p_2$. Later significant improvements has been done to this result by improving the range of $p$  \cite{GrafakosMCZT, KenigMEAFI} and by reducing the smoothness condition on ${\bf m}$ \cite{TomitaAHTMTFMO}.
In 2010's many authors were concerned with weighted inequalities for the bilinear multipliers. In this direction, weighted inequalities with classical $A_p$ weights were proved by Fujita and Tomita \cite{FujitaWNIFMFM} and Hu et. al \cite{HuWNIFMSIOAA} under H\"ormander condition which is weaker than the condition (\ref{coif-myer condition classical setting}). Also, Bui and Duong \cite{BuiWNIFMOAATFMO} and Li and Sun \cite{LiWEFMFM} presented similar results but with multiple weights introduced by Lerner et. al \cite{LernerNMF} in place of the classical weights.
\par In Dunkl setting too there are analogous results \cite{AmriLABMOFTDT, HejnaHMT} to classical Fourier multiplier operators for the non-weighted cases. However, for bilinear multipliers there is a lack of proper analogue to the classical setting even for the non-weighted case. In fact, in Dunkl setting boundedness of bilinear multiplier operators are known only in two special cases. The first case is due to Wr\'{o}bel \cite{WrobelABMVAFC}, where he assumes that the multiplier ${\bf m}$ is radial in both the variables, that is there exists a function ${\bf m}_0$ on $(0, \infty) \times (0, \infty)$ such that ${\bf m}(\xi,\eta)={\bf m}_0(|\xi|,|\eta|)$ and the second one is obtained by Amri et al. \cite{AmriLABMOFTDT}, where they restricted themselves to the one-dimensional case only. This motivates us to address the gap and acquire a suitable counterpart to the classical results for bilinear multipliers in the Dunkl setting.
\par Our first aim in this paper is to prove a Coifman-Meyer type multiplier theorem in Dunkl setting (Theorem \ref{bilinear multi main thm}) without those extra assumptions mentioned above.
The main obstacle that comes in obtaining such results is due to lack of appropriate Littlewood-Paley type theorems in Dunkl setting. Thanks to recent results \cite{HejnaRODTONRK} on pointwise estimates of multipliers, which allows us to overcome such difficulties and present a Littlewood-Paley type theory. Once such tools are available, the rest of our work lies in properly adapting some classical techniques (\cite[pp. 67-71]{MuscaluBook2CAMHA}, \cite[Theorem 4.1]{WrobelABMVAFC}) in Dunkl set up along with some new ideas.
\par Moving forward to the next step, we want to prove one and two-weight inequalities for the bilinear Dunkl-multiplier operators (Theorem \ref{one weight bilinear multi} and Theorem \ref{two weight bilinear multi}) with multiple weights and also for the exponents beyond the Banach range $1<p<\infty$. In our results, the smoothness condition on the multiplier $\bf{m}$  and weight classes are not the same to that of the corresponding results in the classical case (see \cite[Theorem 1.2]{LiWEFMFM} and \cite[Theorem 4.2]{BuiWNIFMOAATFMO}). To prove the boundedness results, the approaches used in the classical setting highly depend on the fact that the two transposes of the operator ${\bf T}_{\bf m}$ are also bilinear multiplier operators with multipliers ${\bf m}(-\xi-\eta, \eta)$ and ${\bf m}(\xi,-\xi-\eta)$. In Dunkl case, no mechanism is known to find the multipliers of the so called transposes of 
the Dunkl bilinear multiplier operator. Moreover, our results are different than the classical case as we have also included the two-weight case and the end-point cases. Hence, a routine adaptation of the classical techniques is not sufficient to achieve these results. In order to obtain these weighted inequalities, an attempt has been made here to develop the theory of multilinear Calder\'on-Zygmund type singular integrals in Dunkl setting which is not available in the literature as of our best knowledge.
\par In classical setting the theory of multilinear singular integrals has its own rich history \cite{CoifmanOCOSIABSI, CoifmenADOPD, DoungMOFMSIWNSK, GrafakosMWNIFMMSIWNSK, GrafakosMCZT, LernerNMF}. But in Dunkl setting, only linear Calder\'on-Zygmund type operators are introduced recently by Tan et. al \cite{TanSIOT1TLPT}. In this paper, we define $m$-linear Dunkl-Calder\'on-Zygmund operators which can be treated as Dunkl counter part of the multilinear Calder\'on-Zygmund operators in classical setting introduced by Grafakos and Torres \cite{GrafakosMCZT}. Our results regrading multilinear Calder\'on-Zygmund operators may be of independent interest and useful in future explorations.
\par Our main results regarding Dunkl-Calder\'on-Zygmund operators are one and two-weight inequalities (Theorem \ref{one weight calderon Zygmund} and Theorem \ref{two weight calderon Zygmund}). In the proofs of these theorems, we closely follow the scheme used in \cite{LernerNMF} (see also \cite{GrafkosMAMS}). However, due to involvement of the action of the reflection group in Definition \ref{defn of multi Dunkl-Calderon _Zygmung oper}, some new ideas regarding $G$-orbits and results for spaces of homogeneous type (see Definition \ref{defn space of homogeneous type}) are required to complete the proofs. Here we also mention that by arguing as in \cite[p. 10]{TanSIOT1TLPT}, we can see that the smoothness conditions (\ref{smoothness estimate of kernel}) assumed on the kernel are weaker than that of the Calder\'on-Zygmund singular integral operators given in spaces of homogeneous type \cite[p. 20]{GrafkosMAMS}. So the results for singular integrals in spaces of homogeneous type do not imply our results. Finally, we show that bilinear multipliers belong to the class of bilinear Dunkl-Calder\'on-Zygmund operators and we conclude the proofs of the weighted inequalities for bilinear multipliers.

\par In the theory of Dunkl analysis, our findings on multilinear singular integral operators pave the way for studying various multilinear operators within the Dunkl framework. Notably, within this article, we have derived that bilinear multipliers represent specific class of such operators. Furthermore, we anticipate the extension of this work to encompass other type of operators, including multilinear Dunkl-pseudo-differential operators belonging to some particular symbol classes, much like in the classical case (\cite{CaoWASTEFTMPDO, GrafakosMCZT}). Also, this, in turn, prompts us to explore other types of multilinear singular integrals, such as rough singular integrals and singular integrals with Dini-type conditions, along with their associated commutators. Moreover, the results concerning bilinear multipliers may have potential applications in establishing fractional Leibniz-type rules for the Dunkl Laplacian and various Kato-Ponce-type inequalities. For instance, in \cite{WrobelABMVAFC}, a fractional Leibniz-type rules for the Dunkl Laplacian  has been proved for the group $\mathbb{Z}_2^d$ using bilinear multiplier theorem with radial multipliers. 
\par The organization of the rest of the paper is as follows. In Section \ref{prel}, we provide some notations and preliminaries of Dunkl setting. Section \ref{Littlewood-Plaey section} is devoted to establish Littlewood-Paley type theorems in Dunkl setting. In Section \ref{homogeneous space section}, we discuss some results on spaces of homogeneous type and on Muckenhoupt weights. We introduce a multilinear set up  in Section \ref{statement main theorem}. Definition and statements of main results (Theorem \ref{two weight calderon Zygmund} and Theorem \ref{one weight calderon Zygmund}) for multilinear Dunkl-Calder\'on-Zygmund operators are given in Section \ref{statemnt of multi Calderon} and main results (Theorem \ref{two weight bilinear multi} and Theorem \ref{one weight bilinear multi}) regarding bilinear Dunkl multiplier operators are given in Section \ref{statemnt of bilinear multi}. Section \ref{proof main theorem singular integral} contains proofs of the weighted inequalities for the Calder\'on-Zygmund operators in Dunkl setting. Finally, in Section \ref{proof main theorem multiplier}, we first state and prove Coifman-Meyer type multiplier theorem in Dunkl setting (Theorem \ref{bilinear multi main thm}). Then using weighted boundedness of the multilinear singular integrals and adapting the recent results of Dziuba\'{n}ski and Hejna \cite{HejnaRODTONRK} to bilinear setting, we conclude the proofs of the weighted inequalities for bilinear multipliers. Throughout the paper $C$ denotes a universal constant (depending on the parameters involved) which may not be the same in every line and sometimes we will write the parameters as suffixes on $C$ to show the dependence of the parameters.

\section{Preliminaries of Dunkl Theory}\label{prel}

Let us consider the usual inner product $\langle \cdot\ ,\cdot\rangle$ and the usual norm $|\cdot| := \sqrt{\langle \cdot\ , \cdot \rangle}$  on $\mathbb{R}^d$. For any $\lambda\in \mathbb{R}^d \setminus \{0\} $, the \emph{reflection} $\sigma_{\lambda}$ with respect to the hyperplane $\lambda^{\bot }$ orthogonal to $\lambda$ is given by
 $$\sigma_{\lambda}(x)=x-2\frac{\langle x,\lambda \rangle}{|\lambda|^2}\lambda.$$ Let $ R$ be a finite subset of $\mathbb{R}^d$ which does not contain 0. If  $R$ satisfies $R \cap \mathbb{R}\lambda = \{ \pm \lambda \} $ for all $\lambda \in R $ and $\sigma _{\lambda} (R) = R $ for all $\lambda \in R$, then $R$ is called a \emph{root system}. Throughout the paper we will consider a fixed  normalized root system $R$, that is $|\lambda|=\sqrt{2},\ \forall \lambda \in R$. The subgroup $G$  generated by reflections $\{ \sigma_ \lambda : \lambda \in R\}$ is called the  \emph{reflection group} (or  \emph{Coxeter group}) associated with $R$ and a $G$-invariant function $k :R\rightarrow \mathbb{C}$, is known as a  \emph{multiplicity function}. In this paper, we take a fixed multiplicity function $k\geq 0$. Let $h_k$ be the $G$-invariant homogeneous weight function given by $h_k(x) = \prod \limits _ {\lambda \in R} | \langle x , \lambda \rangle|^{k (\lambda)}$ and  
$d\mu_k(x)$ be the normalized measure $c_k h_k(x)dx$, where 
$$c_k^{-1}=\int_{\mathbb{R}^d}e^{-{|x|^2}/{2}}\,h_k(x)\,dx,$$
$d_k=d+\gamma_k $ and $\gamma_k=\sum\limits_{\lambda \in R}k(\lambda)$.
\par Let $x\in \mathbb{R}^d$, $r>0$ and $B(x,r)$ be the ball with centre at $x$ and radius $r$. Then the volume $\mu_k(B(x,r))$ of $B(x,r)$ is given by \footnote{ The symbol $\sim$ between two positive expressions means that their ratio remains between two positive constants.}
\begin{equation}\label{Volofball}
    \mu_k(B(x,r))\sim \ r^d\prod\limits_{\lambda \in R}\left(|\langle x,\lambda \rangle|+r\right)^{k(\lambda)}.
\end{equation}
It is immediate from above  that if $r_2>r_1>0$, then
\begin{eqnarray}\label{VOLRADREL}
   C\left(\frac{r_1}{r_2}\right)^{d_k}\leq  \frac{\mu_k(B(x,r_1))}{\mu_k(B(x,r_2))}\leq C^{-1} \left(\frac{r_1}{r_2}\right)^{d}.
\end{eqnarray}
Let $d_G(x,y)$ denote the distance between the $G$-orbits of $x$ and $y$, that is, $d_G(x,y)=\min\limits_{\sigma \in G}|\sigma(x)-y|$ and for any $r>0$, we write $V_G(x,y,r)=\max\,\left\{\mu_k(B(x,r)),\mu_k(B(y,r))\right\}$. Then from the expression for volume of a ball, it follows that
\begin{equation}\label{V(x,y,d(x,y) comparison}
V_G(x,y,d_G(x,y))\sim \mu_k(B(x,d_G(x,y))\sim \mu_k(B(y,d_G(x,y)).
\end{equation}
Let us define the orbit $\mathcal{O}(B)$ of the ball $B$ by 
$$\mathcal{O}(B)=\big\{y\in \mathbb{R}^d: d_G(c_B,y)\leq r(B)\big\}=\bigcup\limits_{\sigma \in G}\sigma (B),$$ 
where $c_B$ denotes the centre and $r(B)$ denotes the radius of the ball $B$. It then follows that
\begin{equation}\label{orbit volume compa}
    \mu_k(B)\leq \mu_k(\mathcal{O}(B))\leq |G|\,\mu_k(B).
\end{equation}
Although $d_G$ satisfies the triangle inequality, it is \emph{not} a metric on $\mathbb{R}^d$. Also, note that for any $x,y\in \mathbb{R}^d$, we have $d_G(x,y)\leq |x-y|.$
\par The \emph{differential-difference operators} or the \emph{Dunkl operators} $T_{\xi}$ introduced by  C.F. Dunkl \cite{Dunkl}, is given by
\begin{eqnarray*}
T_{\xi} f(x)= \partial_{\xi} f(x)+\sum\limits_{\lambda \in R} \frac{k(\lambda )}{2} \langle\lambda, \xi \rangle \frac{f(x) - f(\sigma _\lambda  x)}{\langle \lambda  , x \rangle}.
\end{eqnarray*}
The Dunkl operators $T_{\xi}$ are the $k$-deformations of the directional derivative operators $\partial_{\xi}$ and coincide with them in the case $k=0$.
For a fixed $y\in \mathbb{R}^d$, there is a unique real analytic solution for the system
$T_{\xi} f =  \langle y,\xi \rangle  f $  satisfying $f(0)=1$. The solution $f(x)=E_k(x,y)$ is known as the \emph{Dunkl kernel}. 
For two reasonable functions $f$ and $g$, the following integration by parts formula is well known:
 $$\int_{\mathbb{R}^d} T_{\xi}\,f(x) g(x) \,d\mu_k(x)=-\int_{\mathbb{R}^d}f(x)T_{\xi}\,g(x) \,d\mu_k(x).$$
 Also, if at least one of $f$ or $g$ is $G$-invariant then the Leibniz-type rule holds:
 $$T_\xi(fg)(x)=T_\xi f(x)\, g(x)+f(x)T_\xi g(x).$$
\par Let us consider the canonical orthonormal basis $\{e_j : j = 1, 2, \cdots, d\}$ in $\mathbb{R}^d$ and set $T_j = T_{e_j}$ and $\partial_{j}=\partial_{e_j}$. For any multi-index $\alpha=(\alpha_1,\alpha_2,\cdots, \alpha_d)\in \left(\mathbb{N}\cup\{0\}\right)^d$, we use the following notations.\\
$\bullet$ $|\alpha|=(\alpha_1+\alpha_2+\cdots+ \alpha_d),$\\
$\bullet$ $\partial_j^0=I$, $\partial^\alpha=\partial_1^{\alpha_1}\circ \partial_2^{\alpha_2}\circ\cdots \circ\partial_d^{\alpha_d},$\\
$\bullet$ $T_j^0=I$, $T^\alpha=T_1^{\alpha_1}\circ T_2^{\alpha_2}\circ\cdots \circ T_d^{\alpha_d}.$\\
Sometimes we will write $\partial_{\xi}^{\alpha}$ to indicate that the partial derivatives are taken with respect to the variable $\xi.$
\par The Dunkl kernel $E_k(x,y)$ which actually generalizes the exponential functions $e^{<x,y>}$, has a unique
extension to a holomorphic function on $\mathbb{C}^d\times \mathbb{C}^d$. We list below few properties of the Dunkl kernel (see \cite{RoslerPODIO,RoslerDOTA, RoslerAPRF} for details).\\
$\bullet$ $E_k(x,y)=E_k(y,x)$ for any $x,\ y\in \mathbb{C}^d,$\\
%$\bullet$ $E_k(x,0)=1$ for any $x\in \mathbb{C}^d,$\\
$\bullet$ $E_k(tx,y)=E_k(x,ty)$ for any $x,\ y\in \mathbb{C}^d$ and $t\in \mathbb{C},$\\
$\bullet$ $|\partial^{\alpha}_zE_k(ix,z)|\leq |x|^{|\alpha|}$ for any $x,\ z\in \mathbb{R}^d$ and $\alpha\in \left(\mathbb{N}\cup\{0\}\right)^d.$
\par Let $L^p(\mathbb{R}^d,d\mu_k)$ denote the space of complex valued measurable functions $f$ such that 
$$||f||_{L^p(d\mu_k)}:=\Big(\int_{\mathbb{R}^d}|f(x)|^p\,d\mu_k(x)\Big)^{1/p}<\infty$$
and  $L^{p,\,\infty}(\mathbb{R}^d,d\mu_k)$ be the corresponding weak space with norm
$$||f||_{L^{p,\,\infty}(d\mu_k)}:=\sup\limits_{t>0}\,t\,[\mu_k\big(\{x\in \mathbb{R}^d:|f(x)|>t\}\big)]^{1/p}<\infty.$$
 For any $f$ $\in L^1(\mathbb{R}^d,d\mu_k)$, the \emph{Dunkl transform} of $f$ is defined by
\begin{eqnarray*}
\mathcal{F}_kf(\xi)=\int_{\mathbb{R}^d}f(x)E_k(-i\xi,x)\,d\mu_k(x).
\end{eqnarray*}
The following properties of Dunkl transform are known in the literature \cite{DunklHTATFRG, deJeuTDT}.\\
$\bullet$ $\mathcal{F}_k$ preserves the space $\mathcal{S}(\mathbb{R}^d ),$\\
$\bullet$ $\mathcal{F}_k$ extends to an isometry on  $L^2(\mathbb{R}^d,d\mu_k)$ (Plancherel formula), that is,
$$||\mathcal{F}_kf||_{L^2(d\mu_k)}=||f||_{L^2(d\mu_k)},$$
$\bullet$ If both $f$ and $\mathcal{F}_kf$ are in $L^1(\mathbb{R}^d,d\mu_k)$, then the following Dunkl inversion formula holds
$$f(x)=\mathcal{F}_k^{-1}(\mathcal{F}_kf)(x)=:\int_{\mathbb{R}^d}\mathcal{F}_kf(\xi)E_k(i\xi,x)\,d\mu_k(\xi),$$
$\bullet$ From definition of the Dunkl kernel, for any $f\in \mathcal{S}(\mathbb{R}^d )$, the following relations holds :
$$T_j\mathcal{F}_kf(\xi)=-\mathcal{F}_k(i(\cdot)_jf)(\xi) \text{ and } \mathcal{F}_k(T_jf)(\xi)=i\xi_j \mathcal{F}_kf(\xi).$$

\par The \emph{Dunkl translation} $\tau^k_xf$ of a function $f\in  L^2(\mathbb{R}^d,d\mu_k)$, is defined in \cite{ThangaveluCOMF} in terms of Dunkl transform by 
$$\mathcal{F}_k(\tau^k_xf)(y)=E_k(ix,y)\mathcal{F}_kf(y).$$
Since $E_k(ix,y)$ is bounded, the above formula defines $\tau^k_x$ as a bounded operator on $L^2(\mathbb{R}^d,d\mu_k).$ We collect few properties of the Dunkl translations which will be used later.\\
$\bullet$ For $f\in \mathcal{S}(\mathbb{R}^d )$, $\tau^k_x$ can be pointwise defined as 
$$\tau^k_xf(y)=\int_{\mathbb{R}^d}E_k(ix,\xi)E_k(iy,\xi)\mathcal{F}_kf(\xi)\,d\mu_k(\xi),$$
$\bullet$ $\tau_y^kf(x)=\tau_{x}^kf(y)$ for any $f$ in $\mathcal{S}(\mathbb{R}^d )$,\\
$\bullet$ $\tau^k_x(f_t)=(\tau^k_{t^{-1}x}f)_t,\ \forall x\in\mathbb{R}^d$ and $\forall f\in\mathcal{S}(\mathbb{R}^d)$, where $f_t(x)=t^{-d_k}f(t^{-1}x)$ and $\ t>0$.\\
Although $\tau^k_x$ is bounded operator for radial functions in $L^p(\mathbb{R}^d,d\mu_k)$ (see \cite{RoslerAPRF}), it is not known whether Dunkl translation is bounded operator or not on whole $L^p(\mathbb{R}^d,d\mu_k)$ for $p\neq2.$

\par For $f,g \in L^2(\mathbb{R}^d, d\mu_k)$, the \emph{Dunkl convolution} $f*_kg$
of $f$ and $g$ is defined by 
$$f*_kg(x)=\int_{\mathbb{R}^d}f(y)\tau^k_xg(-y)\,d\mu_k(y).$$
$*_k$ has the following basic properties (see \cite{ThangaveluCOMF} for details).\\
$\bullet$ $f*_kg(x)=g*_kf(x)$ for any $f, g \in  L^2(\mathbb{R}^d, d\mu_k)$;\\
$\bullet$ $\mathcal{F}_k(f*_kg)(\xi)=\mathcal{F}_kf(\xi) \mathcal{F}_kg(\xi)$ for any $f, g \in  L^2(\mathbb{R}^d, d\mu_k)$.\\
\par As mentioned for the Dunkl operators case, for $k=0$ Dunkl transform becomes the Euclidean Fourier transform and the Dunkl translation operator becomes the usual translation. In this sense, Dunkl transform is a generalization of Euclidean Fourier transform and putting $k=0$, we can recover the corresponding results in the classical setting from our results.

\section{Littlewood-Paley Type Theorems in Dunkl Setting}\label{Littlewood-Plaey section}
In this section, we prove two different Littlewood-Paley type theorems which are the main ingredients in the proof of Theorem \ref{bilinear multi main thm}. We start with the following theorem. A particular case \cite[Theorem 5.2]{DaiLRTAMBROFTDT} of this theorem is known for the group $\mathbb{Z}_2^d$ with Muckenhoupt weights. 
\begin{thm}\label{Littlewood-Paley l2 thm}
 Let $u\in \mathbb{R}^d, 1<p<\infty$. Let $\psi$ be a smooth function on $\mathbb{R}^d$ such that $supp\, \psi\subset\{\xi\in \mathbb{R}^d:1/r\leq |\xi|\leq r\}$ for some $r>1$. For $j\in \mathbb{Z}$, define $\psi_j(\xi)=\psi (\xi/2^j)$ and for $f\in \mathcal{S}(\mathbb{R}^d)$ define
$$\psi(u,D_k/2^j)f(x)=\int_{\mathbb{R}^d}\psi _j(\xi) \,e^{i\langle u,\, \xi\rangle /2^j}\mathcal{F}_kf(\xi)E_k(ix,\xi)\,d\mu_k(\xi).$$ Then 
$$ \Big\|\Big(\sum\limits_{j\in \mathbb{Z}}|\psi(u,D_k/2^j)f|^2\Big)^{1/2}\Big\|_{L^p(d\mu_k)}\leq C\,(|1+|u|)^{n} ||f||_{L^p(d\mu_k)},$$
where $n=\lfloor d_k \rfloor+2$ and $C$ is independent of $u$.
\end{thm}
\begin{proof}
We will use the theory of Banach-valued singular integral operators \cite[Theorem 3.1]{AmriSIOIDS} to prove the above theorem. The $L^2$-case follows in similar way to the classical case \cite[p. 160]{DuoanBook}. Using Plancherel formula for Dunkl transform, we get
\begin{eqnarray*}
    \Big\|\Big(\sum\limits_{j\in \mathbb{Z}}|\psi(u,D_k/2^j)f|^2\Big)^{1/2}\Big\|^2_{L^2(d\mu_k)}&=& \int_{\mathbb{R}^d}\sum\limits_{j\in\mathbb{Z}}|\psi_j(x)\,e^{i\langle u,\, \xi\rangle /2^j}|^2|\mathcal{F}_kf(x)|^2\,d\mu_k(x)\\
    &\leq&\int_{\mathbb{R}^d}\sum\limits_{j\in\mathbb{Z}}|\psi_j(x)|^2|\mathcal{F}_kf(x)|^2\,d\mu_k(x)\\
    &\leq& C_r\, ||f||^2_{L^2(d\mu_k)}
\end{eqnarray*}
 where in the last step we have used the fact that for any $x$ only a fixed finite number of $j$'s (depending on $r$) will contribute in the sum. This concludes the $L^2$-case.
 \par Let $\Psi^{u}\in \mathcal{S}(\mathbb{R}^d)$ be such that $\mathcal{F}_k\Psi^{u}(\xi)=\psi(\xi)e^{i\langle u,\, \xi\rangle }$ and define $\Psi^{u}_j(\xi)=2^{jd_k}\Psi^{u}(2^j\xi)$. Then $\mathcal{F}_k\Psi^{u}_j(\xi)=\psi_j(\xi)e^{i\langle u,\, \xi\rangle /2^j}$ and we also have that
$$\Big(\sum\limits_{j\in \mathbb{Z}}|\psi(u,D_k/2^j)f(x)|^2\Big)^{1/2}=\Big(\sum\limits_{j\in \mathbb{Z}}|f*_k \Psi^{u}_j(x)|^2\Big)^{1/2}.$$
Thus to apply the above mentioned Theorem, we only need to show that 
$$\int\limits_{|y-y'|\leq d_G(x,y)/2}\|\tau^k_x\Psi^{u}_j(-y)-\tau^k_x\Psi^{u}_j(-y')\|_{\ell^2(\mathbb{Z})}\,d\mu_k(x)\leq C\,(|1+|u|)^{n}$$
$$\text{and }\int\limits_{|y-y'|\leq d_G(x,y)/2}\|\tau^k_y\Psi^{u}_j(-x)-\tau^k_{y'}\Psi^{u}_j(-x)\|_{\ell^2(\mathbb{Z})}\,d\mu_k(x)\leq C\,(|1+|u|)^{n}.$$
Again to prove the above two inequalities, it is enough to show that for $x,y,y'\in \mathbb{R}^d$ with $|y-y'|\leq d_G(x,y)/2$,
\begin{equation}\label{vectorvalued kernel1}
\|\tau^k_x\Psi^{u}_j(-y)-\tau^k_x\Psi^{u}_j(-y')\|_{\ell^2(\mathbb{Z})}\leq\frac{ C\,(|1+|u|)^{n}}{V_G\left(x,y,d_G(x,y)\right)}\,\frac{|y-y'|}{d_G(x,y)}
\end{equation}
\begin{equation}\label{vectorvalued kernel2}
\|\tau^k_y\Psi^{u}_j(-x)-\tau^k_{y'}\Psi^{u}_j(-x)\|_{\ell^2(\mathbb{Z})}\leq \frac{C\,(|1+|u|)^{n}}{V_G\left(x,y,d_G(x,y)\right)}\,\frac{|y-y'|}{d_G(x,y)}.
\end{equation}

\par We will only prove (\ref{vectorvalued kernel1}), as the proof of (\ref{vectorvalued kernel2}) follows from (\ref{vectorvalued kernel1}) by symmetry.
\par Now using the formula for Dunkl translation and the  definition of $\Psi^u_j$, we have 
\begin{equation}\label{integral representation of the kernel}
\tau^k_x\Psi^{u}_j(-y)=2^{jd_k}\int_{\mathbb{R}^d}\psi(\xi)\,e^{i\langle u,\, \xi\rangle }E_k(i\xi,2^jx)E_k(-i\xi,2^jy)\,d\mu_k(\xi).
\end{equation}
Next, we calculate $\|\psi(\cdot)\,e^{i\langle u,\, \cdot\rangle }\|_{C^n(\mathbb{R}^d)}$. Using usual Leibniz rule for any multi-index $\alpha$, we have 
\begin{eqnarray*}
\partial^{\alpha}\big(\psi(\xi)\,e^{i\langle u,\, \xi\rangle }\big)&=&\sum\limits_{\beta\leq\alpha}\binom {\alpha}{\beta} \partial^{\beta}e^{i\langle u,\, \xi\rangle }\, \partial^{\alpha -\beta} \psi(\xi),
    \end{eqnarray*}
    where the summation ranges over all multi-indices $\beta$ such that $\beta_j\leq\alpha_j$ for all $1\leq j\leq d.$
\par Hence 
    \begin{eqnarray}\label{growth of u}
\|\psi(\cdot)\,e^{i\langle u,\, \cdot\rangle }\|_{C^n(\mathbb{R}^d)}&=&\sup\limits_{\xi\in\mathbb{R}^d,\, |\alpha|\leq n}| \partial^{\alpha}\psi(\xi)\,e^{i\langle u,\, \xi\rangle }|\\
       &\leq& C\, (1+|u|)^n \sup\limits_{|\alpha|\leq n}\|\partial^{\alpha} \psi\|_{L^{\infty}}.\nonumber
    \end{eqnarray}
Next, we estimate $|\tau^k_x\Psi^{u}_j(-y)-\tau^k_{x}\Psi^{u}_j(-y')|$. We calculate the estimate in two parts:
\\\\
\underline{If $|2^jy-2^jy'|\leq 1$:}\\\\
In view of (\ref{integral representation of the kernel}), applying \cite[eq.(4.31)]{HejnaRODTONRK} and using the inequalities (\ref{growth of u}) and (\ref{VOLRADREL}), we have
\begin{eqnarray}\label{less than 1 case  1st estimate}
    &&|\tau^k_x\Psi^{u}_j(-y)-\tau^k_{x}\Psi^{u}_j(-y')|\\
    &\leq&C\, (1+|u|)^n \frac{2^{jd_k}|2^jy-2^jy'|}{\left(\mu_k(B(2^jx,1))\,\mu_k(B(2^jy,1))\right)^{1/2}}\frac{1}{1+2^j |x-y|} \frac{1}{\left(1+2^jd_G(x,y)\right)^{n-1}}\nonumber\\
   &\leq& C\, (1+|u|)^n \frac{|y-y'|}{|x-y|}\frac{1}{\left(\mu_k(B(x,2^{-j}))\,\mu_k(B(y,2^{-j}))\right)^{1/2}} \frac{1}{\left(1+2^jd_G(x,y)\right)^{n-1}}\nonumber
\end{eqnarray}
Now, when $2^jd_G(x,y)\leq 1$, from (\ref{VOLRADREL}) and (\ref{V(x,y,d(x,y) comparison}) we get 
\begin{eqnarray*}
    \frac{1}{\mu_k(B(x,2^{-j}))}&\leq& C\, \big(2^jd_G(x, y)\big)^d \frac{1}{\mu_k(B(x,d_G(x, y)))}\\
    &\leq& C\, \big(2^jd_G(x, y)\big)^d \frac{1}{V_G(x, y, d_G(x, y))}\\
     &\leq& C\,  \frac{\big(2^jd_G(x, y)\big)^d +\big(2^jd_G(x, y)\big)^{d_k}}{V_G(x, y, d_G(x, y))}.
\end{eqnarray*}
Similarly, when $2^jd_G(x,y)> 1$, from (\ref{VOLRADREL}) and (\ref{V(x,y,d(x,y) comparison}) we get 
\begin{eqnarray*}
    \frac{1}{\mu_k(B(x,2^{-j}))}&\leq& C\,  \frac{\big(2^jd_G(x, y)\big)^d +\big(2^jd_G(x, y)\big)^{d_k}}{V_G(x, y, d_G(x, y))}.
\end{eqnarray*}
Thus, in any case
\begin{eqnarray*}
    \frac{1}{\mu_k(B(x,2^{-j}))}&\leq& C\,  \frac{\big(2^jd_G(x, y)\big)^d +\big(2^jd_G(x, y)\big)^{d_k}}{V_G(x, y, d_G(x, y))}.
\end{eqnarray*}
In similar manner we can deduce
\begin{eqnarray*}
    \frac{1}{\mu_k(B(y,2^{-j}))}&\leq& C\,  \frac{\big(2^jd_G(x, y)\big)^d +\big(2^jd_G(x, y)\big)^{d_k}}{V_G(x, y, d_G(x, y))}.
\end{eqnarray*}
Therefore, if $|2^jy-2^jy'|\leq 1$, using above two estimates, from (\ref{less than 1 case  1st estimate})
we write
\begin{eqnarray}\label{less than 1 case  Last estimate}
    &&|\tau^k_x\Psi^{u}_j(-y)-\tau^k_{x}\Psi^{u}_j(-y')|\\
    &\leq& C\, (1+|u|)^n \frac{|y-y'|}{|x-y|}\frac{\big(2^jd_G(x, y)\big)^d +\big(2^jd_G(x, y)\big)^{d_k}}{V_G(x, y, d_G(x, y))} \frac{1}{\left(1+2^jd_G(x,y)\right)^{n-1}}.\nonumber
\end{eqnarray}
\underline{If $|2^jy-2^jy'|> 1$:}\\\\
Again in view of (\ref{integral representation of the kernel}), applying \cite[eq.(4.30)]{HejnaRODTONRK} and using the inequalities (\ref{growth of u}),  (\ref{VOLRADREL}) and (\ref{V(x,y,d(x,y) comparison}) in similar manner as in the last case, we get
\begin{eqnarray}\label{greater than 1 case  1st estimate}
&&|\tau^k_x\Psi^{u}_j(-y)-\tau^k_{x}\Psi^{u}_j(-y')|\\
    &\leq&C\, (1+|u|)^n \Bigg[ \frac{2^{jd_k}}{\left(\mu_k(B(2^jx,1))\,\mu_k(B(2^jy,1))\right)^{1/2}}\frac{1}{1+2^j |x-y|} \frac{1}{\left(1+2^jd_G(x,y)\right)^{n-1}}\nonumber\\
    &&+ \frac{2^{jd_k}}{\left(\mu_k(B(2^jx,1))\,\mu_k(B(2^jy',1))\right)^{1/2}}\frac{1}{1+2^j |x-y'|} \frac{1}{\left(1+2^jd_G(x,y')\right)^{n-1}}\Bigg]\nonumber\\
     &\leq& C\, (1+|u|)^n \Bigg[\frac{\big(2^jd_G(x, y)\big)^d +\big(2^jd_G(x, y)\big)^{d_k}}{V_G(x, y, d_G(x, y))} \frac{1}{1+2^j |x-y|} \frac{1}{\left(1+2^jd_G(x,y)\right)^{n-1}}\nonumber\\
    && + \frac{\big(2^jd_G(x, y')\big)^d +\big(2^jd_G(x, y')\big)^{d_k}}{V_G(x, y', d_G(x, y'))} \frac{1}{1+2^j |x-y'|} \frac{1}{\left(1+2^jd_G(x,y')\right)^{n-1}}\Bigg]\nonumber
\end{eqnarray}
It is easy to see that the condition $|y-y'|\leq d_G(x,y)/2$ implies that 
$$d_G(x, y)\sim d_G(x, y'),\ |x-y|\sim |x-y'|\text{ and }V_G(x, y, d_G(x, y))\sim V_G(x, y', d_G(x, y')).$$
So if $|2^jy-2^jy'|> 1$, applying the above estimates in (\ref{greater than 1 case  1st estimate}) we write
\begin{eqnarray}\label{greater than 1 case  Last estimate}
&&|\tau^k_x\Psi^{u}_j(-y)-\tau^k_{x}\Psi^{u}_j(-y')|\\
&\leq& C\, (1+|u|)^n\frac{\big(2^jd_G(x, y)\big)^d +\big(2^jd_G(x, y)\big)^{d_k}}{V_G(x, y, d_G(x, y))} \frac{1}{1+2^j |x-y|} \frac{1}{\left(1+2^jd_G(x,y)\right)^{n-1}}\nonumber\\
&\leq& C\, (1+|u|)^n\frac{\big(2^jd_G(x, y)\big)^d +\big(2^jd_G(x, y)\big)^{d_k}}{V_G(x, y, d_G(x, y))} \frac{|2^jy-2^jy'|}{1+2^j |x-y|} \frac{1}{\left(1+2^jd_G(x,y)\right)^{n-1}}\nonumber\\
 &\leq& C\, (1+|u|)^n \frac{|y-y'|}{|x-y|}\frac{\big(2^jd_G(x, y)\big)^d +\big(2^jd_G(x, y)\big)^{d_k}}{V_G(x, y, d_G(x, y))} \frac{1}{\left(1+2^jd_G(x,y)\right)^{n-1}}.\nonumber
\end{eqnarray}
\par Now taking (\ref{less than 1 case  Last estimate}) and (\ref{greater than 1 case  Last estimate}) into account and using $|x-y|\geq d_G(x, y)$ together with the condition $n=\lfloor d_k \rfloor+2$, we have
\begin{eqnarray*}
   &&\|\tau^k_x\Psi^{u}_j(-y)-\tau^k_{x}\Psi^{u}_j(-y')\|_{\ell^2(\mathbb{Z})}\\
   &\leq& \sum\limits_{j\in\mathbb{Z}}|\tau^k_x\Psi^{u}_j(-y)-\tau^k_{x}\Psi^{u}_j(-y')|\\
    &\leq& \frac{C\, (1+|u|)^n}{V_G(x, y, d_G(x, y))} \frac{|y-y'|}{d_G(x,y)} \sum\limits_{j\in \mathbb{Z}}\frac{\big(2^jd_G(x, y)\big)^d +\big(2^jd_G(x, y)\big)^{d_k}}{\left(1+2^jd_G(x,y)\right)^{n-1}}.\\
   &\leq& \frac{C\, (1+|u|)^n}{V_G(x, y, d_G(x, y))} \frac{|y-y'|}{d_G(x,y)}\Bigg(\sum\limits_{j\in\mathbb{Z}:\,2^j d_G(x,y)\leq 1}\cdots +\sum\limits_{j\in\mathbb{Z}:\,2^j d_G(x,y)> 1}\cdots \Bigg)\\
   &\leq& \frac{C\, (1+|u|)^n}{V_G(x, y, d_G(x, y))} \frac{|y-y'|}{d_G(x,y)}\Bigg(\sum\limits_{j\in\mathbb{Z}:\,2^j d_G(x,y)\leq 1}\big(2^jd_G(x, y)\big)^d \\
   && +\sum\limits_{j\in\mathbb{Z}:\,2^j d_G(x,y)> 1}\frac{(2^jd_G(x, y))^{d_k}}{(2^jd_G(x,y))^{n-1}} \Bigg)\\
   &\leq& \frac{C\, (1+|u|)^n}{V_G(x, y, d_G(x, y))} \frac{|y-y'|}{d_G(x,y)}.
   \end{eqnarray*}
This completes the proof of (\ref{vectorvalued kernel1}) and hence the proof of the theorem.
\end{proof}
\begin{rem}
In \cite[Theorem 3.1]{AmriSIOIDS} the explicit constant for the boundedness of Banach-valued singular integrals is not calculated. However, a close observation (see \cite[Theorem 3.1]{AmriSIOIDS} and \cite[Theorem 1.1]{GrafakosVVSIAMFOSOHT} ) assures that the constant in our proof will vary as $(1+|u|)^{n}$.
\end{rem}
Let ${\bf m}$ be a bounded function on $\mathbb{R}^d.$ For any $t>0$ and for any $f\in \mathcal{S}(\mathbb{R}^d)$, we define a Dunkl-multiplier operator ${\bf m}_t(D_k)$ as
$${\bf m}_t(D_k)f(x)=\int_{\mathbb{R}^d}{\bf m}(t\xi)\mathcal{F}_kf(\xi)E_k(ix,\xi)\,d\mu_k(\xi).$$
 
 Then, the we have the following boundedness result.
 \begin{prop}\label{max function asso to multi}
   Let ${\bf m}$ be a function on $\mathbb{R}^d$  such that 
   $$|{\bf m}(x)|\leq C_{\bf m}/(1+|x|) \text{ and }|\nabla {\bf m}(x)|\leq C_{{\bf m}'}/(1+|x|) \text{ for all } x\in \mathbb{R}^d,$$
   where $\nabla$ is the usual gradient on $\mathbb{R}^d.$
    Then 
 $$\Big\|\sup\limits_{t>0}|{\bf m}_t(D_k)f|\Big\|_{L^2(d\mu_k)}\leq C(C_{\bf m}+C_{{\bf m}'}) ||f||_{L^2(d\mu_k)}$$
 \end{prop}
 \begin{proof}
      The proof follows by repeating the proof in the classical case \cite[pp. 397-398]{RubiodefranciaMFAFT} with the classical objects replaced by their Dunkl-counterparts.
 \end{proof}
The next main result of this section is as follows.
\begin{thm}\label{Littlewood-Plaey l infinty thm}
    Let $u\in \mathbb{R}^d, 1<p<\infty$. Let $\psi$ be a compactly supported smooth function on $\mathbb{R}^d$. For $j\in \mathbb{Z}$, define $\psi_j(\xi)=\psi (\xi/2^j)$ and for $f\in \mathcal{S}(\mathbb{R}^d)$ define
$$\psi(u,D_k/2^j)f(x)=\int_{\mathbb{R}^d}\psi _j(\xi)\,e^{i\langle u,\, \xi\rangle /2^j}\mathcal{F}_kf(\xi) E_k(ix,\xi)\,d\mu_k(\xi).$$ Then 
$$ \Big\|\sup\limits_{j\in\mathbb{Z}}|\psi(u,D_k/2^j)f|\Big\|_{L^p(d\mu_k)}\leq C\,(|1+|u|)^{n} ||f||_{L^p(d\mu_k)},$$
where $n=\lfloor d_k \rfloor+2$ and $C$ is independent of $u$.
\end{thm}
\begin{proof}
The proof follows in the same scheme as the proof of Theorem \ref{Littlewood-Paley l2 thm}. We will only provide a outline of the proof.
\par Since $\psi$ is a smooth function with compact support, $\partial _j \psi$ is also so and $\partial _j e^{i\langle u,\, \xi\rangle }=iu_j e^{i\langle u,\, \xi\rangle }$ for any $1\leq j \leq d.$ Therefore, the following estimates hold for all $\xi \in \mathbb{R}^d$:
$$\left|\psi (\xi)e^{i\langle u,\, \xi\rangle }\right|\leq \frac{C}{1+|\xi|} \text{ and }\left|\nabla \big(\psi (\xi)e^{i\langle u,\, \xi\rangle }\big)\right|\leq \frac{C(1+|u|)}{1+|\xi|},$$
where $C$ does not depend on $u.$
\par Hence by Proposition \ref{max function asso to multi},
$$\Big\|\sup\limits_{j\in\mathbb{Z}}|\psi(u,D_k/2^j)f|\Big\|_{L^2(d\mu_k)}\leq C(1+|u|) ||f||_{L^2(d\mu_k)}.$$
\par Let $\Psi^u$ be as in the proof of Theorem \ref{Littlewood-Paley l2 thm}. Thus, to complete the proof we only need to prove that  for $x,y,y'\in \mathbb{R}^d$ with $|y-y'|\leq d_G(x,y)/2$,
\begin{equation}\label{vectorvalued kernel1 for l infinity}
\sup\limits_{j\in \mathbb{Z}}|\tau^k_x\Psi^{u}_j(-y)-\tau^k_{x}\Psi^{u}_j(-y')|\leq\frac{ C\,(|1+|u|)^{n}}{V_G\left(x,y,d_G(x,y)\right)}\,\frac{|y-y'|}{d_G(x,y)},
\end{equation}
which follows by repeating the arguments used in the proof of Theorem \ref{Littlewood-Paley l2 thm}.
\end{proof}

\section{Spaces of Homogeneous Type and Muckenhoupt Weights}\label{homogeneous space section}
In this section, we will state some well known results on  homogeneous spaces in the sense of Coifman and Weiss \cite{Coifmanbook} and on Muckenhoupt weights.
\begin{defn}\label{defn space of homogeneous type}
A \emph{space of homogeneous type} $(X, \rho,d\mu)$ is a topological space equipped with a quasi metric $\rho$ and a Borel measure $d\mu $ such that
 \begin{enumerate}[label=(\roman*)]
\item  $\rho$ is continuous on $X\times X$ and the balls $B_{\rho}(x,r)):=\{ y\in X: \rho (x,y)<r \}$ are open in $X$;
\item  the measure $\mu$ fulfills the doubling condition :
$$\mu(B_{\rho}(x,2r))\leq C \mu(B_{\rho}(x,r)),\ \forall x\in X,\ \forall r>0;$$
\item  $0<\mu(B_{\rho}(x,r))<\infty$ for every $x\in X$ and $r>0$.
\end{enumerate}
\end{defn}
It turns out that the measure $\mu_k$ is a Borel measure on $\mathbb{R}^d$ and also (\ref{Volofball}) implies that $\mu_k$ satisfies the doubling condition. Thus, $(\mathbb{R}^d,| x-y |,d\mu_k) $ is a space of homogeneous type.
\par For any $m\in \mathbb{N}$ and given any $\overrightarrow {f}=(f_1,f_2,\cdots,f_m)$, where each $f_j$ is locally integrable on $(X, \rho,d\mu)$, the \emph{multilinear Hardy-Littlewood maximal function} $\mathcal{M}^X_{HL}$ is defined as
$$\mathcal{M}^X_{HL}\overrightarrow {f}(x)=\sup_{ \underset{x\in B_{\rho}}{ B_{\rho}\subset X}}\prod\limits^m_{j=1}   \frac{1}{ \mu(B_{\rho})} \int_{B_{\rho}}|f_j(y)|\,d\mu(y) ,$$
where supremum is taken over all balls $B_{\rho}$ in $X$ which contains $x$. 
\par  On the space $(\mathbb{R}^d,| x-y |,d\mu_k)$, we will use  $\mathcal{M}^k_{HL}$ to denote the multilinear Hardy-Littlewood maximal function, i.e., 
$$\mathcal{M}^k_{HL}\overrightarrow {f}(x)=\sup_{ \underset{x\in B}{ B\subset \mathbb{R}^d}}\prod\limits^m_{j=1}   \frac{1}{ \mu_k(B)} \int_{B}|f_j(y)|\,d\mu_k(y)$$
and for the case $m=1$, we will use the notation 
$$M^k_{HL}f(x)=\sup_{ \underset{x\in B}{ B\subset \mathbb{R}^d}}   \frac{1}{ \mu_k(B)} \int_B|f(y)|\,d\mu_k(y).$$ 
It is known in the literature that $M^k_{HL}$ is bounded on $L^p(\mathbb{R}^d, d\mu_k)$ for $1<p<\infty.$ For any locally integrable function $f$, we also define the \emph{sharp maximal function} $ M^{k,\,\#}_{HL}$, given by
$$M^{k,\,\#}_{HL}f(x)=\sup_{ \underset{x\in B}{ B\subset \mathbb{R}^d}}   \frac{1}{ \mu_k(B)} \int_B|f(y)-f_B|\,d\mu_k(y),$$
where $$f_B=\frac{1}{ \mu_k(B)} \int_B f(y)\,d\mu_k(y).$$
For $\epsilon>0$ set 
$$ M^{k,\,\#}_{HL,\,\epsilon}\,f(x)=\big(M^{k,\,\#}_{HL}(|f|^\epsilon)(x)\big)^{1/\epsilon}.$$
Observe that  for $\epsilon, a, b>0$, the inequalities
$$\min\{1, 2^{\epsilon -1}\}(a^\epsilon +b^\epsilon)\leq (a+b)^\epsilon \leq \max\{1, 2^{\epsilon -1}\}(a^\epsilon +b^\epsilon)$$
yield 
$$M^{k,\,\#}_{HL,\,\epsilon}\,f(x)\sim\sup_{ \underset{x\in B}{ B\subset \mathbb{R}^d}}   \inf\limits_{c\in\,\mathbb{C}}\Big[\frac{1}{ \mu_k(B)} \int_B\big|\,|f(y)|^\epsilon-|c|^\epsilon\big|\,d\mu_k(y)\Big]^{1/\epsilon}.$$
\par A non negative function $w$ on $X$ is called a \emph{weight} if it is locally integrable. Next, we give the definitions of some well-known weight classes.
\begin{defn}
Let $1\leq p<\infty$ and $w$ be a weight. The weight $w$ is said to belong to the class $A_p(X, \rho,d\mu)$, if it satisfies
$$\sup_{B_{\rho}\subset X }\Big(  \frac{1}{ \mu(B_{\rho})} \int_{B_{\rho}} w(y) d\mu(y) \Big)\Big(  \frac{1}{ \mu(B_{\rho})} \int\limits_{B_{\rho}}w(y)^{1-p'}d\mu(y) \Big)^{p-1} < \infty,$$
when $p=1$, $\Big(  \frac{1}{ \mu(B_{\rho})} \int_{B_{\rho}}w(y)^{1-p'}d\mu(y) \Big)^{p-1}$ is understood as $\left(\inf\limits_{B_{\rho}}w\right)^{-1}.$
\end{defn}
Set $A_{\infty}(X, \rho,d\mu):=\bigcup\limits_{1\leq p< \infty} A_p(X, \rho,d\mu).$
\begin{defn}
Let $1\leq p_1,p_2,\cdots,p_m<\infty$, $\overrightarrow P=(p_1,p_2,\cdots,p_m)$ and $p$ be the number given by $1/p=1/p_1+1/p_2+\cdots +1/p_m$. Furthermore, let $v, w_1, w_2,\cdots,w_m$ be weights and $\overrightarrow w =(w_1,w_2,\cdots ,w_m)$. We say that the vector weight $(v,\overrightarrow w)$ is in the class $A_{\overrightarrow P}(X, \rho,d\mu)$, if it satisfies
$$\sup_{B_{\rho}\subset X }\Big(  \frac{1}{ \mu(B_{\rho})} \int_{B_{\rho}}v(y) d\mu(y) \Big)^{{1}/{p}}\, {\prod\limits^m_{j=1}} \Big(  \frac{1}{ \mu(B_{\rho})} \int_{B_{\rho}}w_j(y)^{1-p_j'}d\mu(y) \Big)^{{1}/{p_j'}} < \infty,$$ 
when $p_j=1$, $\Big(  \frac{1}{ \mu(B_{\rho})} \int_{B_{\rho}}w_j(y)^{1-p_j'}d\mu(y) \Big)^{{1}/{p_j'}}$ is understood as $\left(\inf\limits_{B_{\rho}}w_j\right)^{-1}$.
\end{defn}
In particular when $v=\prod\limits_{j=1}^m w_j^{p/p_j}$, we will simply say that $\overrightarrow w$ is in the class $A_{\overrightarrow P}(X, \rho,d\mu)$.
 \par For the homogeneous space $(\mathbb{R}^d,| x-y |,d\mu_k) $, we will use the notations $A^k_{\infty}$ and $A^k_{\overrightarrow P}$ instead of $A_{\infty}(\mathbb{R}^d,| x-y |,d\mu_k)$ and $A_{\overrightarrow P}(\mathbb{R}^d,| x-y |,d\mu_k)$ respectively.  In our main theorems, we will restrict ourselves to $G$-invariant weights only, i.e., weights $w$ in $\mathbb{R}^d$ which satisfies $w(\sigma (x))=w(x),\ \forall x\in \mathbb{R}^d$ and $\forall \sigma \in G$. In fact $G$-invariance of the Dunkl measure motivates us to use weights which are $G$-invariant. As in \cite{DaiLRTAMBROFTDT, SumanWIFMFOIDS}, using such weights, we will take the benefit of the fact that for a reasonable function $f$ on $\mathbb{R}^d$, for any $\sigma \in G$,
$$\int_{\mathbb{R}^d}f\circ \sigma(x)w(x)\,d\mu_k(x)=\int_{\mathbb{R}^d}f(x)w(x)\,d\mu_k(x).$$
\par  We end this section by stating two theorems which follow from well known results for general space of homogeneous type \cite[Theorem 4.4, Theorem 4.6 and Theorem 4.7]{GrafkosMAMS}.
\begin{thm}\label{multi maximal funct two weight}
  Let $1\leq p_1,p_2,\cdots,p_m<\infty$, $\overrightarrow P=(p_1,p_2,\cdots,p_m)$, $p$ be the number given by $1/p=1/p_1+1/p_2+\cdots +1/p_m$ and $v, w_1, w_2,\cdots,w_m$ be weights. Then the following hold:
  \begin{enumerate}[label=(\roman*)]
\item\label{multi maximal two weigh weak thm}  if $p_j= 1 $ for some $1\leq j \leq m$ and 
the vector weight $(v,\overrightarrow w)\in A_{\overrightarrow P}^k$, then for all $\overrightarrow {f} \in L^{p_1}(\mathbb{R}^d,w_1\,d\mu_k)\times L^{p_2}(\mathbb{R}^d,w_2\,d\mu_k)\times\cdots \times L^{p_m}(\mathbb{R}^d,w_m\, d\mu_k)$, the following boundedness holds:
        $$\sup\limits_{t>0}\,t\Big(\int\limits_{ \{y\,\in{\mathbb{R}^d:\,{M}_{HL} ^k \overrightarrow {f} (y)>t\}}}v(x)\,d\mu_k(x)\Big)^{{1}/{p}}\leq C\  {\prod\limits^m_{j=1}} \Big(   \int_{\mathbb{R}^d}|f_j(x)|^{p_j}w_j(x)d\mu_k(x) \Big)^{{1}/{p_j}};$$
\item \label{multi maximal two weight strong thm} if $p_j> 1 $ for all $1\leq j \leq m$ and 
for some $t>1$ the vector weight $(v,\overrightarrow w)$ 
satisfies the bump condition
\begin{equation}\label{bump condt multi ap weights}
  \sup_{B\subset \mathbb{R}^d }\Big(  \frac{1}{ \mu_k(B)} \int_{B}v(y) d\mu_k(y) \Big)^{{1}/{p}} {\prod\limits^m_{j=1}} \Big(  \frac{1}{ \mu_k(B)} \int_{B}w_j(y)^{-tp_j'/p_j}d\mu_k(y) \Big)^{{1}/{tp_j'}} < \infty,
\end{equation}
then for all $\overrightarrow {f} \in L^{p_1}(\mathbb{R}^d,w_1\,d\mu_k)\times L^{p_2}(\mathbb{R}^d,w_2\,d\mu_k)\times\cdots \times L^{p_m}(\mathbb{R}^d,w_m\, d\mu_k)$, the following boundedness holds:
        $$\Big(\int_{\mathbb{R}^d}\big(\mathcal{M}_{HL} ^k \overrightarrow {f} (x)\big)^p v(x)\, d\mu_k(x) \Big)^{{1}/{p}}\leq C\  {\prod\limits^m_{j=1}} \Big(   \int_{\mathbb{R}^d}|f_j(x)|^{p_j}w_j(x)\,d\mu_k(x) \Big)^{{1}/{p_j}}.$$

\end{enumerate}
\end{thm}
\begin{thm}\label{multi maximal funct one weight}
 Let $1< p_1,p_2,\cdots,p_m<\infty$, $\overrightarrow P=(p_1,p_2,\cdots,p_m)$, $p$ be the number given by $1/p=1/p_1+1/p_2+\cdots +1/p_m$ and $w_1, w_2,\cdots,w_m$ be weights such that the vector weight $\overrightarrow w\in A_{\overrightarrow P}^k$; then for all $\overrightarrow{f} \in L^{p_1}(\mathbb{R}^d,w_1\,d\mu_k)\times L^{p_2}(\mathbb{R}^d,w_2\,d\mu_k)\times\cdots \times L^{p_m}(\mathbb{R}^d,w_m\, d\mu_k)$, the following boundedness holds:
         $$\Big(\int_{\mathbb{R}^d}\big(\mathcal{M}_{HL} ^k \overrightarrow {f} (x)\big)^p \prod\limits_{j=1}^m w_j(x)^{p/p_j}\, d\mu_k(x) \Big)^{{1}/{p}}\leq C\  {\prod\limits^m_{j=1}} \Big(   \int_{\mathbb{R}^d}|f_j(x)|^{p_j}w_j(x)\,d\mu_k(x) \Big)^{{1}/{p_j}}.$$
\end{thm}

\section{Statements of Main Theorems}\label{statement main theorem}
We start this section by extending the Dunkl theory to a multilinear set up.
Let $R$ be the root system and $k$ be the multiplicity function as in Section \ref{prel}. Then
$$R^m:=(R\times(0)_{m-1})\cup \big((0)_1\times R\times(0)_{m-2}\big)\cup \cdots \cup \big((0)_{m-1}\times R\big),$$ where $(0)_j=\{(0,0,\cdots,0)\}\subset (\mathbb{R}^d)^j$, defines a root system in $\big(\mathbb{R}^d \big)^m$. The reflection group acting on $(\mathbb{R}^{d})^m$ is isomorphic to the $m$-fold product $G\times G\times\cdots \times G$.
Let $k^m : R^m\rightarrow \mathbb{C}$ be defined by $$k^m((0,0,\cdots,\lambda ,\cdots,0))=k(\lambda) \text{ for any } \lambda \in R.$$
Then it follows that $k^m$ is a non-negative normalized multiplicity function on $R^m$.
Due to this choice of the Root system and multiplicity function, Dunkl objects on $(\mathbb{R}^{d})^m$ splits into product of the corresponding objects in $\mathbb{R}^d$.
In fact, using the notations as described in Section \ref{prel}, it follows that for any $x_1,y_1,x_2,y_2,\cdots,x_m,y_m \in \mathbb{R}^d$, 
 we have $$d\mu_{k^m}\big((x_1,x_2,\cdots,x_m)\big)=d\mu_k(x_1)d\mu_k(x_2)\cdots d\mu_k(x_m).$$
Again the structures of the root system and the multiplicity function allow us to write
$$E_{k^m}\big((x_1,x_2,\cdots,x_m),(y_1,y_2,\cdots ,y_m)\big)=E_k(x_1,y_1)E_k(x_2,y_2) \cdots E_k(x_m,y_m).$$ 
This at once implies that for reasonable functions $f_1,f_2,\cdots,f_m$;
$$\mathcal{F}_{k^m}\left(f_1 \otimes f_2 \otimes \cdots \otimes f_m\right)\big((x_1,x_2,\cdots, x_m)\big)=\mathcal{F}_kf_1(x_1)\mathcal{F}_kf_2(x_2)\cdots \mathcal{F}_kf_m(x_m)$$
$$\text{and } \Large{\tau}^{k^m}_{(x_1,x_2,\cdots, x_m)}\left(f_1  \otimes \cdots \otimes f_m\right)\big((y_1,y_2,\cdots, y_m)\big)=\tau_{x_1}^kf_1(y_1)\,\cdots\tau_{x_m}^kf_m(y_m).$$
Also the $m$-fold counterparts of all the properties mentioned in Section \ref{prel} hold in this case.
\subsection{Multilinear Calder\'on-Zygmund type Singular Integrals in Dunkl setting}\label{statemnt of multi Calderon}
Motivated by the multilinear Calder\'on-Zygmund operators \cite{GrafakosMCZT}, we define the following multilinear singular operators in Dunkl setting.
\begin{defn}\label{defn of multi Dunkl-Calderon _Zygmung oper}
  An $m$-linear \emph{Dunkl-Calder\'on-Zygmund operator} is a function $\mathcal{T}$ defined on the $m$-fold product $\mathcal{S}(\mathbb{R}^d)\times \mathcal{S}(\mathbb{R}^d)\times \cdots \times \mathcal{S}(\mathbb{R}^d)$ and taking values on $\mathcal{S}'(\mathbb{R}^d)$ such that for all $f_j\in C_c^{\infty}(\mathbb{R}^d)$ with $\sigma (x)\notin \bigcap\limits_{j=1}^m supp\,f_j$ for all  $\sigma\in G$, $\mathcal{T}$ can be represented as 
  $$\mathcal{T}(\overrightarrow {f})(x)=\int_{(\mathbb{R}^{d})^m}K(x,y_1,y_2,\cdots,y_m)\prod\limits_{j=1}^m f_j(y_j)\,d\mu_k(y_j),$$
  where $K$ is a function defined away from the set $\mathcal{O}(\bigtriangleup _{m+1})$
  $$=: \big\{(x,y_1,y_2,\cdots,y_m)\in (\mathbb{R}^d)^{m+1}:x= \sigma_j (y_j) \text{ for some }\sigma_j \in G,\text{ for all }1\leq j \leq m \big\}$$
  which satisfies the following size estimate and smoothness estimates for some $0<\epsilon \leq 1$:
  \begin{equation}\label{size estimate of kernel}
    |K(y_0,y_1,y_2,\cdots,y_m)| \leq  C_K\,\big[\sum\limits_{j=1}^m \mu_k\big(B(y_0,d_G(y_0,y_j))\big)\big]^{-m} \left[\frac{\sum\limits_{j=1}^m d_G(y_0,y_j)}{\sum\limits_{j=1}^m |y_0-y_j|}\right]^{\epsilon},
  \end{equation}
 for all $(y_0,y_1,y_2,\cdots,y_m)\in (\mathbb{R}^d)^{m+1}\setminus \mathcal{O}(\bigtriangleup _{m+1})$;
 \begin{eqnarray}\label{smoothness estimate of kernel}
     &&|K(y_0,y_1,y_2,\cdots,y_n,\cdots, y_m)- K(y_0,y_1,y_2,\cdots,y'_n,\cdots,y_m)|\\
     &\leq& C_K \,\big[\sum\limits_{j=1}^m \mu_k\big(B(y_0,d_G(y_0,y_j))\big)\big]^{-m} 
     \left[\frac{|y_n-y'_n|}{\max\limits_{1\leq j\leq m} |y_0-y_j|}\right]^{\epsilon},\nonumber
  \end{eqnarray}
  whenever $|y_n-y'_n|\leq \max\limits_{1\leq j\leq m} d_G(y_0,y_j)/2$, for all $n\in \{0,1,\cdots,m\}$.
\end{defn}
Note that the size condition (\ref{size estimate of kernel}) guarantees that the above integral is convergent and hence pointwise $\mathcal{T}(\overrightarrow f)$ makes sense. Also, this definition of  multilinear Calder\'on-Zygmund operators in Dunkl setting  matches with the definition of multilinear Calder\'on-Zygmund operators in classical setting \cite{GrafakosMCZT} as well as with the definition of linear  Calder\'on-Zygmund operators in Dunkl setting \cite{TanSIOT1TLPT}.
Our main results regarding multilinear Dunkl-Calder\'on-Zygmund operators are the following two-weight and one-weight inequalities:
\begin{thm}\label{two weight calderon Zygmund}
 Let $1\leq p_1,p_2,\cdots,p_m<\infty$, $\overrightarrow P=(p_1,p_2,\cdots,p_m)$, $p$ be the number given by $1/p=1/p_1+1/p_2+\cdots +1/p_m$ and $v, w_1, w_2,\cdots,w_m$ be $G$-invariant weights with $v\in A^k_{\infty}$. Furthermore let $\mathcal{T}$ maps from $L^{q_1}(\mathbb{R}^d,\,d\mu_k)\times L^{q_2}(\mathbb{R}^d,\,d\mu_k)\times\cdots \times L^{q_m}(\mathbb{R}^d,\, d\mu_k)$ to $L^{q,\,\infty}(\mathbb{R}^d,\, d\mu_k)$ with norm $A$ for some $q,q_1,q_2,\cdots, q_m$ satisfying $1\leq q_1,q_2,\cdots, q_m<\infty$ with $1/q=1/q_1+1/q_2+\cdots +1/q_m$. Then the following hold:
  \begin{enumerate}[label=(\roman*)]
  \item\label{CZ two weigh weak thm}  if $p_j= 1 $ for some $1\leq j \leq m$ and 
the vector weight $(v,\overrightarrow w)\in A_{\overrightarrow P}^k$, then for all $\overrightarrow {f} \in L^{p_1}(\mathbb{R}^d,w_1\,d\mu_k)\times L^{p_2}(\mathbb{R}^d,w_2\,d\mu_k)\times\cdots \times L^{p_m}(\mathbb{R}^d,w_m\, d\mu_k)$, the following boundedness holds:
\begin{eqnarray*}
    \sup\limits_{t>0}\,t\Big(\int\limits_{ \{y\,\in{\mathbb{R}^d:\,|\mathcal{T} \overrightarrow {f} (y)|>t\}}}v(x)\,d\mu_k(x)\Big)^{{1}/{p}}\leq C(C_K+A){\prod\limits^m_{j=1}} \Big(   \int_{\mathbb{R}^d}|f_j(x)|^{p_j}w_j(x)d\mu_k(x) \Big)^{{1}/{p_j}};
        \end{eqnarray*}
        \item\label{CZ two weight strong thm} if $p_j> 1 $ for all $1\leq j \leq m$ and 
the vector weight $(v,\overrightarrow w)$ 
satisfies the bump condition (\ref{bump condt multi ap weights}) for some $t>1$, then for all $\overrightarrow {f} \in L^{p_1}(\mathbb{R}^d,w_1\,d\mu_k)\times L^{p_2}(\mathbb{R}^d,w_2\,d\mu_k)\times\cdots \times L^{p_m}(\mathbb{R}^d,w_m\, d\mu_k)$, the following boundedness holds:
\begin{eqnarray*}
        \Big(\int_{\mathbb{R}^d}|\mathcal{T}\overrightarrow {f} (x)|^p v(x)\, d\mu_k(x) \Big)^{{1}/{p}}\leq C(C_K+A)\,{\prod\limits^m_{j=1}} \Big(   \int_{\mathbb{R}^d}|f_j(x)|^{p_j}w_j(x)\,d\mu_k(x) \Big)^{{1}/{p_j}}.
        \end{eqnarray*}
    \end{enumerate}
\end{thm}
\begin{thm}\label{one weight calderon Zygmund}
    Let $1\leq p_1,p_2,\cdots,p_m<\infty$, $\overrightarrow P=(p_1,p_2,\cdots,p_m)$, $p$ be the number given by $1/p=1/p_1+1/p_2+\cdots +1/p_m$ and $w_1, w_2,\cdots,w_m$ be $G$-invariant weights and the vector weight $\overrightarrow w\in A_{\overrightarrow P}^k$. Furthermore let $\mathcal{T}$ maps from $L^{q_1}(\mathbb{R}^d,\,d\mu_k)\times L^{q_2}(\mathbb{R}^d,\,d\mu_k)\times\cdots \times L^{q_m}(\mathbb{R}^d,\, d\mu_k)$ to $L^{q,\,\infty}(\mathbb{R}^d,\, d\mu_k)$ with norm $A$ for some $q,q_1,q_2,\cdots, q_m$ satisfying $1\leq q_1,q_2,\cdots, q_m<\infty$ with $1/q=1/q_1+1/q_2+\cdots +1/q_m$. Then the following hold:
    \begin{enumerate}[label=(\roman*)]
  \item\label{CZ one weigh weak thm}  if $p_j= 1 $ for some $1\leq j \leq m$, then for all $\overrightarrow {f} \in L^{p_1}(\mathbb{R}^d,w_1\,d\mu_k)\times L^{p_2}(\mathbb{R}^d,w_2\,d\mu_k)\times\cdots \times L^{p_m}(\mathbb{R}^d,w_m\, d\mu_k)$, the following boundedness holds:
\begin{eqnarray*}
   && \sup\limits_{t>0}\,t\Big(\int\limits_{ \{y\,\in{\mathbb{R}^d:\,|\mathcal{T} \overrightarrow {f} (y)|>t\}}}\prod\limits_{j=1}^m w_j(x)^{p/p_j}\,d\mu_k(x)\Big)^{{1}/{p}}\\
    &\leq& C(C_K+A)\,{\prod\limits^m_{j=1}} \Big(   \int_{\mathbb{R}^d}|f_j(x)|^{p_j}w_j(x)d\mu_k(x) \Big)^{{1}/{p_j}};
\end{eqnarray*}
        
        \item\label{CZ one weigh strong thm} if $p_j> 1 $ for all $1\leq j \leq m$, then for all $\overrightarrow {f} \in L^{p_1}(\mathbb{R}^d,w_1\,d\mu_k)\times L^{p_2}(\mathbb{R}^d,w_2\,d\mu_k)\times\cdots \times L^{p_m}(\mathbb{R}^d,w_m\, d\mu_k)$, the following boundedness holds:
        \begin{eqnarray*}
        &&\Big(\int_{\mathbb{R}^d}|\mathcal{T}\overrightarrow {f} (x)|^p \prod\limits_{j=1}^m w_j(x)^{p/p_j}\, d\mu_k(x) \Big)^{{1}/{p}}\\
        &\leq& C(C_K+A)\,{\prod\limits^m_{j=1}} \Big(   \int_{\mathbb{R}^d}|f_j(x)|^{p_j}w_j(x)\,d\mu_k(x) \Big)^{{1}/{p_j}}.
        \end{eqnarray*}
    \end{enumerate}
\end{thm}
\subsection{Bilinear Multipliers Operators in Dunkl setting}\label{statemnt of bilinear multi}
In order to keep the presentation simple, in this section we will limit ourselves to the case $m=2$ only. However, it is not too hard to extend our results to any natural number $m>2$.
\par In analogy to the classical case, for a bounded function $\bf m$ on $\mathbb{R}^d\times \mathbb{R}^d$ define the \emph{bilinear Dunkl multiplier} $\mathcal{T}_{\bf m}$ as 
$$\mathcal{T}_{\bf m}(f_1,f_2)(x)=\int_{\mathbb{R}^{2d}}{\bf m}(\xi,\eta)\mathcal{F}_kf_1(\xi)\mathcal{F}_kf_2(\eta)E_k(ix,\xi)E_k(ix,\eta)\,d\mu_k(\xi)d\mu_k(\eta)$$
for all $f_1,f_2\in \mathcal{S}(\mathbb{R}^d)$.
Next we state our main results for bilinear multiplier operators.
\begin{thm}\label{two weight bilinear multi}
    Let $1\leq p_1,p_2<\infty$, $p$ be the number given by $1/p=1/p_1+1/p_2$ and $L\in \mathbb{N}$ be such that $L>2d+2\lfloor d_k \rfloor+4$. If ${\bf m}\in C^L\left(\mathbb{R}^d\times \mathbb{R}^d \setminus\{(0,0)\}\right)$ be a function satisfying
 \begin{equation}\label{usual deriv condtn m}
     |\partial^{\alpha}_{\xi}\partial^{\beta}_{\eta}{\bf m}(\xi,\eta)|\leq C_{\alpha,\,\beta}\left(|\xi|+|\eta|\right)^{-(|\alpha|+|\beta|)}
 \end{equation}
 for all multi-indices $\alpha,\beta \in \left(\mathbb{N}\cup\{0\}\right)^d$ such that $|\alpha|+|\beta|\leq L$ and for all $(\xi,\eta)\in \mathbb{R}^d \times \mathbb{R}^d\setminus\{(0,0)\}$ and $v, w_1, w_2$ be $G$-invariant weights with $v\in A^k_{\infty}$;
 then the following hold:
 \begin{enumerate}[label=(\roman*)]
 \item\label{bi multi two weigh weak thm}  if at least one of $p_1$ or $p_2$ is $1$ and 
the vector weight $(v,(w_1,w_2))\in A^k_{(p_1,p_2)}$, then for all $f_1 \in L^{p_1}(\mathbb{R}^d,w_1\,d\mu_k)$ and $f_2\in L^{p_2}(\mathbb{R}^d,w_2\,d\mu_k)$, the following boundedness holds:
\begin{eqnarray*}
    \sup\limits_{t>0}\,t\Big(\int\limits_{ \{y\,\in{\mathbb{R}^d:\,|\mathcal{T}_{\bf m} \overrightarrow {f} (y)|>t\}}}v(x)\,d\mu_k(x)\Big)^{{1}/{p}}
    \leq C\,{\prod\limits^2_{j=1}} \Big(   \int_{\mathbb{R}^d}|f_j(x)|^{p_j}w_j(x)d\mu_k(x) \Big)^{{1}/{p_j}};
    \end{eqnarray*}
     \item\label{bi multi two weigh strong thm}  if both $p_1,p_2> 1 $ and 
the vector weight $(v,\overrightarrow w)$ 
satisfies the bump condition (\ref{bump condt multi ap weights}) with $m=2$ for some $t>1$, then for all $f_1 \in L^{p_1}(\mathbb{R}^d,w_1\,d\mu_k)$ and $f_2\in L^{p_2}(\mathbb{R}^d,w_2\,d\mu_k)$, the following boundedness holds:
\begin{eqnarray*}
        \Big(\int_{\mathbb{R}^d}|\mathcal{T}_{\bf m}\overrightarrow {f} (x)|^p v(x)\, d\mu_k(x) \Big)^{{1}/{p}}\leq C\,{\prod\limits^2_{j=1}} \Big(   \int_{\mathbb{R}^d}|f_j(x)|^{p_j}w_j(x)\,d\mu_k(x) \Big)^{{1}/{p_j}}.
         \end{eqnarray*}
 \end{enumerate}
\end{thm}
\begin{thm}\label{one weight bilinear multi}
    Let $1\leq p_1,p_2<\infty$ and $p$ be the number given by $1/p=1/p_1+1/p_2$ and $L\in \mathbb{N}$ be such that $L>2d+2\lfloor d_k \rfloor+4$. If ${\bf m}\in C^L\left(\mathbb{R}^d\times \mathbb{R}^d \setminus\{(0,0)\}\right)$ be a function which satisfies the condition (\ref{usual deriv condtn m}) for all multi-indices $\alpha,\beta \in \left(\mathbb{N}\cup\{0\}\right)^d$ such that $|\alpha|+|\beta|\leq L$ and for all $(\xi,\eta)\in \mathbb{R}^d \times \mathbb{R}^d\setminus\{(0,0)\}$ and $w_1, w_2$ be $G$-invariant weights with $(w_1,w_2)\in A^k_{(p_1,\,p_2)}$;
 then the following hold:
 \begin{enumerate}[label=(\roman*)]
 \item\label{bi multi one weigh weak thm}   if at least one of $p_1$ or $p_2$ is $1$, then for all $f_1 \in L^{p_1}(\mathbb{R}^d,w_1\,d\mu_k)$ and $f_2\in L^{p_2}(\mathbb{R}^d,w_2\,d\mu_k)$, the following boundedness holds:
\begin{eqnarray*}
    \sup\limits_{t>0}\,t\Big(\int\limits_{ \{y\,\in{\mathbb{R}^d:\,|\mathcal{T}_{\bf m} \overrightarrow {f} (y)|>t\}}}\prod\limits_{j=1}^2 w_j^{p/p_j}(x)\,d\mu_k(x)\Big)^{{1}/{p}}
    \leq C\,{\prod\limits^2_{j=1}} \Big(   \int_{\mathbb{R}^d}|f_j(x)|^{p_j}w_j(x)d\mu_k(x) \Big)^{{1}/{p_j}};
    \end{eqnarray*}
     \item\label{bi multi one weigh strong thm}  if both $p_1,p_2> 1 $, then for all $f_1 \in L^{p_1}(\mathbb{R}^d,w_1\,d\mu_k)$ and $f_2\in L^{p_2}(\mathbb{R}^d,w_2\,d\mu_k)$, the following boundedness holds:
\begin{eqnarray*}
        \Big(\int_{\mathbb{R}^d}|\mathcal{T}_{\bf m}\overrightarrow {f} (x)|^p w_1^{p/p_1}(x)w_2^{p/p_2}(x)\, d\mu_k(x) \Big)^{{1}/{p}}
        \leq C\,{\prod\limits^2_{j=1}} \Big(   \int_{\mathbb{R}^d}|f_j(x)|^{p_j}w_j(x)\,d\mu_k(x) \Big)^{{1}/{p_j}}.
         \end{eqnarray*}
 \end{enumerate}
\end{thm}

\begin{rem}
In our result, smoothness condition on ${\bf m}$ is more than what one may expect from the viewpoint of the classical case \cite[Corollary 1.2]{TomitaAHTMTFMO}. We do not know whether the above result can be proven by assuming less number of derivatives on ${\bf m}$. 
\end{rem}

\section{Proofs of The Weighted Inequalities for Singular Integrals }\label{proof main theorem singular integral}
Throughout this section we will assume that $\mathcal{T}$ is an $m$-linear Dunkl-Calder\'on-Zygmund operator and $\mathcal{T}$ maps from $L^{q_1}(\mathbb{R}^d,\,d\mu_k)\times L^{q_2}(\mathbb{R}^d,\,d\mu_k)\times\cdots \times L^{q_m}(\mathbb{R}^d,\, d\mu_k)$ to $L^{q,\,\infty}(\mathbb{R}^d,\, d\mu_k)$ with norm $A$ for some $q,q_1,q_2,\cdots, q_m$ satisfying $1\leq q_1,q_2,\cdots, q_m<\infty$ with $1/q=1/q_1+1/q_2+\cdots +1/q_m$.
From this a priori boundedness condition, we can prove the following 
theorem.
\begin{thm}\label{1,1...1 to 1/m bddness}
    Let $\mathcal{T}$ be a multilinear operator as described above. Then $\mathcal{T}$ extends to a bounded operator from the $m$-fold product $L^{1}(\mathbb{R}^d,\,d\mu_k)\times L^{1}(\mathbb{R}^d,\,d\mu_k)\times\cdots \times L^{1}(\mathbb{R}^d,\, d\mu_k)$ to $L^{1/m,\,\infty}(\mathbb{R}^d,\, d\mu_k)$ with norm $\leq C(C_K+A)$.
\end{thm}
\begin{proof}
 By density argument, it is enough to show the result for functions in $\mathcal{S}(\mathbb{R}^d).$ Take $f_1,f_2,\cdots,f_m \in \mathcal{S}(\mathbb{R}^d)$ and fix $\alpha>0.$ Define 
 $$E_\alpha=\big\{x\in \mathbb{R}^d:|\mathcal{T}(f_1,f_2,\cdots,f_m)(x)|>\alpha\big\}.$$
 Also, without loss of generality we can take $\|f_j\|_{L^{1}(d\mu_k)}=1$. It then suffices to prove that $\mu_k(E_\alpha)\leq C(C_K+A)^{1/m}\,{\alpha}^{-1/m}$.
\par Let $\nu>0$ be constant to be defined later. Applying Calder\'on-Zygmund decomposition to each of the functions $f_j$ at height $(\alpha\nu)^{1/m}$, we obtain $m$ number of good functions $g_j$, $m$ number of bad functions $b_j$ and $m$ families $I_j$ of balls $\big\{B_{j,\,n}:n\in I_j\big\}$ such that  $f_j=g_j+b_j$, $b_j=\sum\limits_{n\in I_j}b_{j,\,n}$ with the properties that for 
all $n\in I_j$ and $s\in[1,\infty),$
\begin{enumerate}[label=(\roman*)]
    \item $supp\, b_{j,\,n}\subseteq B_{j,\,n}$ and $\int_{\mathbb{R}^d}b_{j,\,n}(y)\,d\mu_k(y)=0;$
    \item $\|b_{j,\,n}\|_{L^{1}(d\mu_k)}\leq C (\alpha\nu)^{1/m}\,\mu_k(B_{j,\,n})$;
    \item $\sum\limits_{n\in I_j}\mu_k(B_{j,\,n})\leq C (\alpha\nu)^{-1/m}$;
    \item $\|g_j\|_{L^{\infty}}\leq C(\alpha\nu)^{1/m}$, $\|g_j\|_{L^{s}(d\mu_k)}\leq C(\alpha\nu)^{1/m(1-1/s)}$ and  $\|b_j\|_{L^{1}(d\mu_k)}\leq C$.
\end{enumerate}
Now let 
\begin{eqnarray*}
    E_1&=&\{x\in \mathbb{R}^d:|\mathcal{T}(g_1,g_2,\cdots,g_m)(x)|>\alpha/2^m\},\\
     E_2&=&\{x\in \mathbb{R}^d:|\mathcal{T}(b_1,g_2,\cdots,g_m)(x)|>\alpha/2^m\},\\
      E_3&=&\{x\in \mathbb{R}^d:|\mathcal{T}(g_1,b_2,\cdots,g_m)(x)|>\alpha/2^m\}\\
     \vdots\\
    \text{ and }   E_{2^m}&=&\{x\in \mathbb{R}^d:|\mathcal{T}(b_1,b_2,\cdots,b_m)(x)|>\alpha/2^m\}.
\end{eqnarray*}
Then $\mu_k\big(\big\{x\in \mathbb{R}^d:|\mathcal{T}(f_1,f_2,\cdots,f_m)(x)|>\alpha\big\}\big)\leq \sum\limits_{n=1}^{2^m}\mu_k(E_n)$. It is now enough to show that for all $n\in\{1,2,\cdots, 2^m\}$,
\begin{equation}\label{mu k E n less than}
  \mu_k(E_n)\leq C(C_K+A)^{1/m} \alpha^{-1/m}.
\end{equation}
Applying the given hypothesis on $\mathcal{T}$, we get 
\begin{eqnarray}\label{estimate for first set}
    \mu_k(E_1)&\leq& C (2^m A/\alpha)^q||g_1||^q_{L^{q_1}(d\mu_k)}||g_2||^q_{L^{q_2}(d\mu_k)}\cdots||g_m||^q_{L^{q_m}(d\mu_k)}\\
    &\leq& C\,A^q\alpha^{-q}(\alpha\nu)^{(q/m)(m-1/q)}\leq C\,A^q \alpha^{-1/m}\nu^{q-1/m}\nonumber.
\end{eqnarray}
Next, let us take $E_n$, where $2\leq n\leq 2^m.$ Consider the case where there are exactly $l$ bad functions appearing in $\mathcal{T}(h_1,h_2,\cdots,h_m)$, where $h_j$ is either $g_j$ or $b_j$ and also let $j_1,j_2,\cdots,j_l$ are the indices which corresponds to the bad functions. We will prove that 
\begin{equation}\label{for n greater 2 mu k En less than }
  \mu_k(E_n)\leq C \alpha^{-1/m}[\nu^{-1/m}+\nu^{-1/m}(\nu\,C_K)^{1/l}] .
\end{equation}
{\color{brown}Let} $r_{j,n}$ be the radius and $c_{j,n}$ be the centre of the ball $B_{j,n}$. Define $(B_{j,n})^{*}=B(c_{j,n},\,2r_{j,n})$ and $(B_{j,n})^{**}=B(c_{j,n},\,5r_{j,n})$. Now 
\begin{eqnarray*}
    \mu_k\big(\bigcup\limits_{j=1}^m\bigcup\limits_{n\in I_j}\mathcal{O}\big((B_{j,n})^{**}\big)\big)&\leq& \sum\limits_{j=1}^m\sum\limits_{n\in I_j}\mu_k\big(\mathcal{O}\big((B_{j,n})^{**}\big)\big)\\
    &\leq& |G|\sum\limits_{j=1}^m\sum\limits_{n\in I_j}\mu_k\big((B_{j,n})^{**}\big)\\
    &\leq& C\sum\limits_{j=1}^m\sum\limits_{n\in I_j}\mu_k\big(B_{j,n})\big)\leq C(\alpha\nu)^{-1/m}.
\end{eqnarray*}
Thus in view of the above inequality, to prove (\ref{for n greater 2 mu k En less than }), we only need to show that 
\begin{eqnarray}\label{mu k outside orbit}
    &&\mu_k\big(\big\{x\notin \bigcup\limits_{j=1}^m \bigcup\limits_{n\in I_j} \mathcal{O}\big((B_{j,n})^{**}\big): |T(h_1,h_2,\cdots,h_m)(x)|>\alpha/2^m\big\}\big)\\
    &\leq& C\, (\alpha\nu)^{-1/m}(C_K\,\nu)^{1/l}.\nonumber
\end{eqnarray}
Fix $x\notin \bigcup\limits_{j=1}^m \bigcup\limits_{n\in I_j} \mathcal{O}\big((B_{j,n})^{**}\big)$. Then 
\begin{eqnarray}\label{defn of Hn1,n2,..}
\ \ |T(h_1,h_2,\cdots,h_m)(x)|&\leq& \sum\limits_{n_1\in I_{j_1}}\cdots \sum\limits_{n_l\in I_{j_l}}\big|\int_{\big(\mathbb{R}^d\big)^m}K(x,y_1,y_2,\cdots,y_m)\\
    &&\times \prod\limits_{s\,\notin \{j_1,\cdots,j_l\}}g_s(y_s)\prod\limits_{s=1}^lb_{j_s,\,n_s}(y_{j_s}) \,d\mu_k(y_1)\cdots\,d\mu_k(y_m)\big|\nonumber\\
    &=:& \sum\limits_{n_1\in I_{j_1}}\cdots \sum\limits_{n_l\in I_{j_l}} H_{n_1,n_2,\cdots,n_l}.\nonumber
\end{eqnarray}
Let us fix balls $B_{j_1,n_1}, B_{j_2,n_2},\cdots,B_{j_l,n_l}$ and without loss of generality {\color{brown} take}
$$r_{j_1,n_1}=\min\limits_{1\leq s\leq l}r_{j_s,n_s}.$$

Then using the smoothness condition (\ref{smoothness estimate of kernel}), we have
\begin{eqnarray}\label{one time use applying smoothness condtn}
&&\ \big|\int_{(B_{j_1,n_1})^{*}}K(x,y_1,y_2,\cdots,y_m)\,b_{j_1,n_1}(y_{j_1})\,d\mu_k(y_{j_1})\big|\\
&=& \big|\int_{(B_{j_1,n_1})^{*}}[K(x,y_1,\cdots,y_{j_1},\cdots,y_m)-K(x,y_1,\cdots,c_{j_1,n_1},\cdots,y_m)]\nonumber\\
&&\times b_{j_1,n_1}(y_{j_1})\,d\mu_k(y_{j_1})\big|\nonumber\\
&\leq& C_K\, \int_{(B_{j_1,n_1})^{*}}  \big[\sum\limits_{n=1}^m \mu_k\big(B(x,d_G(x,y_n))\big)\big]^{-m} 
     \Big[\frac{|y_{j_1}-c_{j_1,n_1}|}{\max\limits_{1\leq n\leq m} |x-y_n|}\Big]^{\epsilon}
 \,|b_{j_1,n_1}(y_{j_1})|\,d\mu_k(y_{j_1})\nonumber.
\end{eqnarray}
Before proceeding further, we need the following lemma.  
\begin{lem}\label{one time use lemma volume comparison}
For any $\epsilon>0$ and $N\in \mathbb{N}$ there exists a constant $C>0$ such that for any $s\in\{1,2,\cdots,N\}$ and $x,y_n\in \mathbb{R}^d$,
$$\int_{\mathbb{R}^d}  \big[\sum\limits_{n=1}^N \mu_k\big(B(x,d_G(x,y_n))\big)\big]^{-1} 
     [{\max\limits_{1\leq n\leq N} d_G(x,y_n)}]^{-\epsilon}\,d\mu_k(y_s)\leq C\,[{\max\limits_{\underset{n\neq s}{1\leq n\leq N} }d_G(x,y_n)}]^{-\epsilon}.$$
     \end{lem}
\begin{proof}
Take $t=\max\limits_{\underset{n\neq s}{1\leq n\leq N} }d_G(x,y_n)$.\\
\underline{Estimate for $d_G(x,y_s)<2t$.}\\
\par In this case using (\ref{VOLRADREL}), it is not hard to see that 
$$\sum\limits_{n=1}^N \mu_k\big(B(x,d_G(x,y_n))\big)\sim\mu_k(B(x,t)).$$
Then from (\ref{orbit volume compa}),
\begin{eqnarray*}
 &&\int_{d_G(x,y_s)<2t}  \big[\sum\limits_{n=1}^N \mu_k\big(B(x,d_G(x,y_n))\big)\big]^{-1} 
     [{\max\limits_{1\leq n\leq N} d_G(x,y_n)}]^{-\epsilon}\,d\mu_k(y_s)\\
     &\leq & C\,  [{\max\limits_{\underset{n\neq s}{1\leq n\leq N} }d_G(x,y_n)}]^{-\epsilon}\int_{d_G(x,y_s)<2t}\frac{1}{\mu_k(B(x,t))} \,d\mu_k(y_s)\\
     &\leq & C\,  [{\max\limits_{\underset{n\neq s}{1\leq n\leq N} }d_G(x,y_n)}]^{-\epsilon}.
\end{eqnarray*}
\underline{Estimate for $d_G(x,y_s)\geq2t$.}\\
\par In this case we have 
$$\sum\limits_{n=1}^N \mu_k\big(B(x,d_G(x,y_n))\big)\sim\mu_k(B(x,d_G(x, y_s))).$$
Hence, similarly applying (\ref{orbit volume compa}) again,
\begin{eqnarray*}
 &&\int_{d_G(x,y_s)\geq 2t}  \big[\sum\limits_{n=1}^N \mu_k\big(B(x,d_G(x,y_n))\big)\big]^{-1} 
     [{\max\limits_{1\leq n\leq N} d_G(x,y_n)}]^{-\epsilon}\,d\mu_k(y_s)\\
     &\leq & C\, \int_{d_G(x,y_s)\geq 2t}  \frac{1}{\mu_k\big(B(x,d_G(x,y_s))\big)} 
     [ d_G(x,y_s)]^{-\epsilon}\,d\mu_k(y_s) \\
     &\leq&C\,\sum\limits_{r=1}^{\infty}\int\limits_{2^rt\leq d_G(x,y_s)<2^{r+1}t}  \frac{1}{\mu_k\big(B(x,d_G(x,y_s))\big)} 
     [ d_G(x,y_s)]^{-\epsilon}\,d\mu_k(y_s) \\
     &\leq & C\,\sum\limits_{r=1}^{\infty}\frac{1}{(2^rt)^{\epsilon}}\frac{|G|\,\mu_k(B(x,2^rt))}{\mu_k(B(x,2^rt))}\leq C\, t^{-\epsilon}
     = C\, [{\max\limits_{\underset{n\neq s}{1\leq n\leq N} }d_G(x,y_n)}]^{-\epsilon}.
\end{eqnarray*}
This completes the proof of the lemma.
\end{proof}
Returning to the proof of Theorem \ref{1,1...1 to 1/m bddness}, taking integration on both sides of (\ref{one time use applying smoothness condtn}) with respect to $y_s\in \{1,2,\cdots,m\}\setminus \{j_1,j_2,\cdots,j_l\}$ and using Lemma \ref{one time use lemma volume comparison} $(m-l)$ times, we get
\begin{eqnarray}\label{before discussion ineq}
 &&\int_{(\mathbb{R}^{d})^{m-l}}\big|\int_{(B_{j_1,n_1})^{*}}K(x,y_1,y_2,\cdots,y_m)\,b_{j_1,n_1}(y_{j_1})\,d\mu_k(y_{j_1})\big| \prod\limits_{s\,\notin \{j_1,j_2,\cdots,j_l\}} d\mu_k(y_s)\\
 &\leq& C_K \, \int_{(B_{j_1,n_1})^{*}}|b_{j_1,n_1}(y_{j_1})|\,\Big[\int_{(\mathbb{R}^{d})^{m-l}} \big[\sum\limits_{n=1}^m \mu_k\big(B(x,d_G(x,y_n))\big)\big]^{-m} 
     \Big[\frac{|y_{j_1}-c_{j_1,n_1}|}{\max\limits_{1\leq n\leq m} d_G(x,y_n)}\Big]^{\epsilon}\nonumber\\
     &&\times \prod\limits_{s\notin\,\{j_1,j_2,\cdots,j_l\} }\,d\mu_k(y_s) \Big]\,d\mu_k(y_{j_1})\nonumber\\
     &\leq& CC_K \, \int_{(B_{j_1,n_1})^{*}}|b_{j_1,n_1}(y_{j_1})|
     \big[\sum\limits_{s=1}^l \mu_k\big(B(x,d_G(x,y_{j_s}))\big)\big]^{-l}
     |y_{j_1}-c_{j_1,n_1}|^{\epsilon}\nonumber\\
     &&\times [{\max\limits_{1\leq s\leq l} d_G(x,y_{j_s})}]^{-\epsilon}\,d\mu_k(y_{j_1}).\nonumber\\
     &\leq& CC_K \,\int_{(B_{j_1,n_1})^{*}}\Big[\frac{r_{j_1,n_1}}{\max\limits_{1\leq s\leq l}d_G(x,y_{j_s})}\Big]^{\epsilon}\big[\sum\limits_{s=1}^l\mu_k\big(B(x,d_G(x,y_{j_s}))\big)\big]^{-l} \,|b_{j_1,n_1}(y_{j_1})| \,d\mu_k(y_{j_1}).\nonumber
\end{eqnarray}
Now $y_{j_s}\in B_{j_s,n_s}$ and  $x\notin \bigcup\limits_{j=1}^m \bigcup\limits_{n\in I_j} \mathcal{O}\big((B_{j,n})^{**}\big)$ together implies 
$$d_G(x,y_{j_s})\sim d_G(x,c_{j_s, n_s})\text{ and hence }\mu_k\big(B(x,d_G(x,y_{j_s}))\big)\sim \mu_k\big(B(x,d_G(x,c_{j_s, n_s}))\big).$$
Also the minimality of $r_{j_1,n_1}$ implies 
$$r_{j_1,n_1}\leq \prod\limits_{s=1}^l(r_{j_s,n_s})^{1/l}.$$
Similarly 
$$\max\limits_{1\leq s\leq l}d_G(x,c_{j_s,n_s})\geq \prod\limits_{s=1}^l[d_G(x,c_{j_s,n_s})]^{1/l}$$ 
$$ \text{ and }\sum\limits_{s=1}^l \mu_k\big(B(x,d_G(x,c_{j_s,n_s}))\big)\geq \prod\limits_{s=1}^l[\mu_k\big(B(x,d_G(x,c_{j_s,n_s}))\big)]^{1/l}.$$
Now taking the above discussions in account, from (\ref{before discussion ineq}) we can write
\begin{eqnarray*}
     &&\int_{(\mathbb{R}^{d})^{m-l}}\big|\int_{(B_{j_1,n_1})^{*}}K(x,y_1,y_2,\cdots,y_m)\,b_{j_1,n_1}(y_{j_1})\,d\mu_k(y_{j_1})\big| \prod\limits_{s\,\notin \{j_1,j_2,\cdots,j_l\}} d\mu_k(y_s)\\
     &\leq& CC_K \,\Big[\frac{r_{j_1,n_1}}{\max\limits_{1\leq s\leq l}d_G(x,c_{j_s,n_s})}\Big]^{\epsilon}\big[\sum\limits_{s=1}^l\mu_k\big(B(x,d_G(x,c_{j_s,n_s}))\big)\big]^{-l} \,\|b_{j_1,n_1}\|_{L^1(d\mu_k)}\\
      &\leq& CC_K \,\|b_{j_1,n_1}\|_{L^1(d\mu_k)}\prod\limits_{s=1}^l\Big[\frac{r_{j_s,n_s}}{d_G(x,c_{j_s,n_s})}\Big]^{\epsilon/l}\frac{1}{\mu_k\big(B(x,d_G(x,c_{j_s,n_s}))\big)}.
\end{eqnarray*}
So using properties of Calder\'on-Zygmund decomposition, from (\ref{defn of Hn1,n2,..}) we write
\begin{eqnarray*}
     H_{n_1,n_2,\cdots,n_l}
     &\leq& \int_{\big(\mathbb{R}^d\big)^{m-1}}\big|\int_{(B_{j_1,n_1})^{*}}K(x,y_1,y_2,\cdots,y_m)\,b_{j_1,n_1}\,d\mu_k(y_{j_1})\big|\\
     &&\times \prod\limits_{s\,\notin \{j_1,\cdots,j_l\}}|g_s(y_s)|\,d\mu_k(y_s)\prod\limits_{s=2}^l\,|b_{j_s,\,n_s}(y_{j_s})|\,d\mu_k(y_{j_s})\\
    &\leq& CC_K\, (\alpha\nu)^{(m-l)/m}\,\|b_{j_1,n_1}\|_{L^1(d\mu_k)}\int_{\big(\mathbb{R}^d\big)^{l-1}}\prod\limits_{j=2}^l\,|b_{j_s,n_s}(y_{j_s})|\,d\mu_k(y_{j_s})  \\
    &&\times\prod\limits_{s=1}^l\Big[\frac{r_{j_1,n_1}}{d_G(x,c_{j_s,n_s})}\Big]^{\epsilon/l}\frac{1}{\mu_k\big(B(x,d_G(x,c_{j_s,n_s}))\big)}\\
    &\leq& CC_K\, (\alpha\nu)^{(m-l)/m}\,\prod\limits_{s=1}^l\,\Big[\frac{r_{j_s,n_s}}{d_G(x,c_{j_s,n_s})}\Big]^{\epsilon/l}\frac{\|b_{j_s,n_s}\|_{L^1(d\mu_k)}}{\mu_k\big(B(x,d_G(x,c_{j_s,n_s}))\big)} \\
    &\leq& CC_K\, \alpha\nu\,\prod\limits_{s=1}^l\,\Big[\frac{r_{j_s,n_s}}{d_G(x,c_{j_s,n_s})}\Big]^{\epsilon/l}\frac{\mu_k(B_{j_s,n_s})}{\mu_k\big(B(x,d_G(x,c_{j_s,n_s}))\big)}.\\
\end{eqnarray*}
Thus for any $x\notin \bigcup\limits_{j=1}^m \bigcup\limits_{n\in I_j} \mathcal{O}\big((B_{j,n})^{**}\big)$, substituting the above inequality in (\ref{defn of Hn1,n2,..}), we have
\begin{eqnarray}\label{one time use before maximal function}
   && |T(h_1,h_2,\cdots,h_m)(x)|\\
   &\leq& CC_K\, \alpha\nu\,\sum\limits_{n_1\in I_{j_1}}\cdots \sum\limits_{n_l\in I_{j_l}}\prod\limits_{s=1}^l\,\Big[\frac{r_{j_s,n_s}}{d_G(x,c_{j_s,n_s})}\Big]^{\epsilon/l}\frac{\mu_k(B_{j_s,n_s})}{\mu_k\big(B(x,d_G(x,c_{j_s,n_s}))\big)}\nonumber\\
    &\leq& CC_K\, \alpha\nu\,\prod\limits_{s=1}^l\,\Bigg[\sum\limits_{n_s\in I_{j_s}}\Big[\frac{r_{j_s,n_s}}{d_G(x,c_{j_s,n_s})}\Big]^{\epsilon/l}\frac{\mu_k(B_{j_s,n_s})}{\mu_k\big(B(x,d_G(x,c_{j_s,n_s}))\big)}\Bigg].\nonumber
\end{eqnarray}
Now using the facts that $d_G(x,c_{j_s,n_s})\sim d_G(x,c_{j_s,n_s})+r_{j_s,n_s}$, $\mu_k\big(B(x,d_G(x,c_{j_s,n_s}))\big)\sim \mu_k\big(B(x,d_G(x,c_{j_s,n_s})+r_{j_s,n_s})\big)$ and the condition (\ref{VOLRADREL}), we get 
\begin{eqnarray}\label{radius voulume less than maximal funcn}
   &&\Big[\frac{r_{j_s,n_s}}{d_G(x,c_{j_s,n_s})}\Big]^{\epsilon/l}\frac{\mu_k(B_{j_s,n_s})}{\mu_k\big(B(x,d_G(x,c_{j_s,n_s}))\big)}\\
   &\leq&C\,\Big[\frac{\mu_k\big(B_{j_s,n_s}\big)}{\mu_k\big(B(x,d_G(x,c_{j_s,n_s})+r_{j_s,n_s})\big)}\Big]^{\epsilon/(l\,d_k)}\frac{\mu_k(B_{j_s,n_s})}{\mu_k\big(B(x,d_G(x,c_{j_s,n_s})+r_{j_s,n_s})\big)}\nonumber\\
   &\leq& C\,\Bigg[ \frac{1}{\mu_k\big(B(x,d_G(x,c_{j_s,n_s})+r_{j_s,n_s})\big)} \int_{\mathbb{R}^d}\chi_{B_{j_s,n_s}}(y)\,d\mu_k(y)
 \Bigg]^{1+\epsilon/(l\,d_k)}\nonumber\\
 &\leq& C\,\Bigg[ \frac{1}{\mu_k\big(B(x,\,d_G(x,c_{j_s,n_s})+r_{j_s,n_s})\big)}\nonumber\\
 &&\times\int\limits_{\mathcal{O}\big(B(x,\,d_G(x,c_{j_s,n_s})+r_{j_s,n_s})\big)}\chi_{B_{j_s,n_s}}(y)\,d\mu_k(y)
 \Bigg]^{1+\epsilon/(l\,d_k)}\nonumber\\
 &\leq& C\,\Bigg[ \sum\limits_{\sigma\in G}\frac{1}{\mu_k\big(B(\sigma(x),d_G(x,\,c_{j_s,n_s})+r_{j_s,n_s})\big)}\nonumber\\
 &&\times\int\limits_{B(\sigma(x),\,d_G(x,c_{j_s,n_s})+r_{j_s,n_s})}\chi_{B_{j_s,n_s}}(y)\,d\mu_k(y)
 \Bigg]^{1+\epsilon/(l\,d_k)}\nonumber\\
 &\leq& C\,\Big[\sum\limits_{\sigma\in G} M^k_{HL}\big(\chi_{B_{j_s,n_s}}\big)(\sigma(x))\Big]^{1+\epsilon/(l\,d_k)}\nonumber.
\end{eqnarray}
Finally, using Chebyshev's inequality, (\ref{one time use before maximal function}), (\ref{radius voulume less than maximal funcn}), H\"older's inequality, $L^{1+\epsilon/(l\,d_k)}$ boundedness of $M^k_{HL}$ and properties of Calder\'on-Zygmund decomposition, we obtain
\begin{eqnarray*}
    &&\mu_k\big(\big\{x\notin \bigcup\limits_{j=1}^m \bigcup\limits_{n\in I_j} \mathcal{O}\big((B_{j,n})^{**}\big): |T(h_1,h_2,\cdots,h_m)(x)|>\alpha/2^m\big\}\big)\\
    &\leq& C\,\alpha^{-1/l}\int\limits_{\mathbb{R}^d\setminus\bigcup\limits_{j=1}^m \bigcup\limits_{n\in I_j} \mathcal{O}\big((B_{j,n})^{**} } |T(h_1,h_2,\cdots,h_m)(x)|^{1/l}\,d\mu_k(x)\\
    &\leq& C\,(C_K\nu)^{1/l}\int_{\mathbb{R}^d}\Bigg[ \prod\limits_{s=1}^l\,\sum\limits_{n_s\in I_{j_s}}\Big[\sum\limits_{\sigma\in G} M^k_{HL}\big(\chi_{B_{j_s,n_s}}\big)(\sigma(x))\Big]^{1+\epsilon/(l\,d_k)} \Bigg]^{1/l}d\mu_k(x)\\
    &\leq& C\,(C_K\nu)^{1/l}\prod\limits_{s=1}^l\,\Bigg[\int_{\mathbb{R}^d} \sum\limits_{n_s\in I_{j_s}}\Big[\sum\limits_{\sigma\in G} M^k_{HL}\big(\chi_{B_{j_s,n_s}}\big)(\sigma(x))\Big]^{1+\epsilon/(l\,d_k)} d\mu_k(x)\Bigg]^{1/l}\\
    &\leq& C\,(C_K\nu)^{1/l}\prod\limits_{s=1}^l\,\Bigg[ \sum\limits_{n_s\in I_{j_s}}\sum\limits_{\sigma\in G}\int_{\mathbb{R}^d}\Big[ M^k_{HL}\big(\chi_{B_{j_s,n_s}}\big)(\sigma(x))\Big]^{1+\epsilon/(l\,d_k)} d\mu_k(x)\Bigg]^{1/l}\\
     &=& C\,(C_K\nu)^{1/l}\prod\limits_{s=1}^l\,\Bigg[ \sum\limits_{n_s\in I_{j_s}}\sum\limits_{\sigma\in G}\int_{\mathbb{R}^d}\Big[ M^k_{HL}\big(\chi_{B_{j_s,n_s}}\big)(x)\Big]^{1+\epsilon/(l\,d_k)} d\mu_k(x)\Bigg]^{1/l}\\
     &\leq& C\,(C_K\nu)^{1/l}\prod\limits_{s=1}^l\,\Big[ \sum\limits_{n_s\in I_{j_s}}|G|\,\mu_k\big(B_{j_s,n_s}\big)\Big]^{1/l}\leq C\,(C_K\nu)^{1/l}(\alpha\nu)^{-1/m}.
    \end{eqnarray*}
This completes the proof of the inequality (\ref{mu k outside orbit}). Finally choosing $\nu=1/(C_K+A)$, from (\ref{estimate for first set}) and  (\ref{for n greater 2 mu k En less than }), we see that (\ref{mu k E n less than}) holds and hence the proof is concluded.
\end{proof}
 We next prove two propositions which will be very useful in the proofs of Theorem \ref{two weight calderon Zygmund} and Theorem \ref{one weight calderon Zygmund}.
\begin{prop}\label{prop sharp maxi funct domi by multi maxi}
   Let $1\leq p_1,p_2,\cdots,p_m<\infty$ and $\nu\in(0,1/m)$. Then there is a constant $C>0$ depending only on $\nu, m, \epsilon$ and $p_j$'s such that for all  $\overrightarrow {f} \in L^{p_1}(\mathbb{R}^d,\,d\mu_k)\times L^{p_2}(\mathbb{R}^d,\,d\mu_k)\times\cdots \times L^{p_m}(\mathbb{R}^d,\, d\mu_k)$,
   \begin{equation*}
      M^{k,\,\#}_{HL,\,\nu} \big(\mathcal{T}\overrightarrow {f}\big)(x)\leq C\, (C_K+A)\sum\limits_{\underset{\sigma _{n_s}\in G}{(n_1,n_2,\cdots,n_m)}} \mathcal{M}^k_{HL} \left(f_1\circ \sigma_{n_1}, f_2\circ \sigma_{n_2},\cdots,f_m\circ \sigma_{n_m}\right)(x).
   \end{equation*}
\end{prop}
\begin{proof}
Fix a ball $B$ such that $x\in B$. From the definition given in Section \ref{homogeneous space section}, it suffices to prove that there is a $\mathfrak{c}_B \in\mathbb{C}$ depending only on $B$ such that 
\begin{eqnarray}\label{esti of sharp maxi domi by multi maxi}
    &&\Big[\frac{1}{ \mu_k(B)} \int_B\big|\,\mathcal{T}\overrightarrow {f}(z)-\mathfrak{c}_B\big|^{\nu}\,d\mu_k(z)\Big]^{1/\nu}\\
    &\leq& C\, (C_K+A)\sum\limits_{\underset{\sigma _{n_s}\in G}{(n_1,n_2,\cdots,n_m)}} \mathcal{M}^k_{HL} \left(f_1\circ \sigma_{n_1}, f_2\circ \sigma_{n_2},\cdots,f_m\circ \sigma_{n_m}\right)(x).\nonumber
\end{eqnarray}
  Let $c_B$ denotes the centre and $r(B)$ denote the radius of the ball $B$ and set $B^{**}=\big(B(c_B,5r(B))\big)$. Also define $f_j^0=f_j\chi_{\mathcal{O}(B^{**})}$ and $f_j^\infty=f_j-f_j^0$. Then 
  \begin{eqnarray*}
    \prod\limits_{j=1}^m f_j(y_j)&=& \prod\limits_{j=1}^m\,\big[f_j^0(y_j)+f_j^\infty(y_j)\big]\\
    &=&\sum\limits_{\alpha_1,\alpha_2,\cdots,\alpha_m\in\,\{0,\infty\}}f_1^{\alpha_1}(y_1)f_2^{\alpha_2}(y_2)\cdots f_m^{\alpha_m}(y_m)\\
     &=&\prod\limits_{j=1}^m f_j^0(y_j) + \sum\limits_{\text{ at lest one }\alpha_n\neq\,0}f_1^{\alpha_1}(y_1)f_2^{\alpha_2}(y_2)\cdots f_m^{\alpha_m}(y_m)\\
  \end{eqnarray*}
Denote $(f_1^0,f_2^0,\cdots,f_m^0)$ by $\overrightarrow{f^0}$, then we have
\begin{equation}\label{dividing T into two parts}
\mathcal{T}\overrightarrow{f}(z)=\mathcal{T}\overrightarrow{f^0}(z)+\sum\limits_{\text{ at lest one }\alpha_n\neq\,0}\mathcal{T}\big(f_1^{\alpha_1},f_2^{\alpha_2},\cdots, f_m^{\alpha_m}\big)(z).
\end{equation}
$$ \text{Set }N=\prod\limits_{j=1}^m\frac{1}{\mu_k(B)}\|f^0_j\|_{L^1(d\mu_k)}\text{ and }\mathfrak{c}_B=\sum\limits_{\text{ at lest one }\alpha_n\neq\,0}\mathcal{T}\big(f_1^{\alpha_1},f_2^{\alpha_2},\cdots, f_m^{\alpha_m}\big)(x).$$
Then from (\ref{dividing T into two parts}), we can write
\begin{eqnarray}\label{2nd time diving T into two parts}
    &&\Big[\frac{1}{ \mu_k(B)} \int_B\big|\,\mathcal{T}\overrightarrow {f}(z)-\mathfrak{c}_B\big|^{\nu}\,d\mu_k(z)\Big]^{1/\nu}\\
    &\leq&C\,\Big[\frac{1}{ \mu_k(B)} \int_B\big|\,\mathcal{T}\overrightarrow {f^0}(z)\big|^{\nu}\,d\mu_k(z)\Big]^{1/\nu}
     +C\,\Big[\frac{1}{ \mu_k(B)} \int_B \sum\limits_{\text{ at lest one }\alpha_n\neq\,0}\nonumber\\
     &&\times\big|\,\mathcal{T}\big(f_1^{\alpha_1},f_2^{\alpha_2},\cdots, f_m^{\alpha_m}\big)(z)-\mathcal{T}\big(f_1^{\alpha_1},f_2^{\alpha_2},\cdots, f_m^{\alpha_m}\big)(x)\big|^{\nu}\,d\mu_k(z)\Big]^{1/\nu}\nonumber.
\end{eqnarray}
Now we consider the first term. It follows from Theorem \ref{1,1...1 to 1/m bddness} that
\begin{eqnarray}\label{1st part domi by multi maxi}
    &&\Big[\frac{1}{ \mu_k(B)} \int_B\big|\,\mathcal{T}\overrightarrow {f^0}(z)\big|^{\nu}\,d\mu_k(z)\Big]^{1/\nu}\\
    &=&\Big[\frac{1}{ \mu_k(B)}\int_{0}^{N(C_K+A)}\nu\,t^{\nu-1}\mu_k\big(\{z\in B:|\mathcal{T}\overrightarrow{f^0}(z)|>t\}\big)\,dt\nonumber\\   
    &&+\frac{1}{ \mu_k(B)}\int_{N(C_K+A)}^\infty  \nu\,t^{\nu-1}\mu_k\big(\{z\in B:|\mathcal{T}\overrightarrow{f^0}(z)|>t\}\big)\,dt\Big]^{1/\nu}\nonumber\\
    &\leq&C \Big[ N^\nu(C_K+A)^{\nu} +(C_K+A)^{1/m}N^{1/m} \int_{N(C_K+A)}^\infty  \nu\,t^{\nu-1-1/m} \,dt\Big]^{1/\nu}\nonumber\\
    &\leq&C (C_K+A)N\leq C\, (C_K+A)\mathcal{M}^k_{HL}\overrightarrow{f}(x)\nonumber\\
    &\leq& C\, (C_K+A)\sum\limits_{\underset{\sigma _{n_s}\in G}{(n_1,n_2,\cdots,n_m)}} \mathcal{M}^k_{HL} \left(f_1\circ \sigma_{n_1}, f_2\circ \sigma_{n_2},\cdots,f_m\circ \sigma_{n_m}\right)(x).\nonumber
\end{eqnarray}
For the second term in (\ref{2nd time diving T into two parts}), we take $\alpha_{j_1}=\alpha_{j_2}=\cdots=\alpha_{j_l}=0$, where for $0\leq s\leq l,\ j_s\in\{1,2,\cdots,m\}$ and $0\leq l\leq m$ with the convention that $\{{j_1},{j_2},\cdots,{j_l}\}=\emptyset $ if $l=0$. Then, keeping in mind that  $m-l\geq 1$, $x\in B$ and $supp\, f_j^\infty\subset\mathbb{R}^d\setminus {\mathcal{O}(B^{**})} $, for any $z\in B$
we can apply smoothness condition (\ref{smoothness estimate of kernel}) to obtain
\begin{eqnarray*}
   &&\big|\,\mathcal{T}\big(f_1^{\alpha_1},f_2^{\alpha_2},\cdots, f_m^{\alpha_m}\big)(z)-\mathcal{T}\big(f_1^{\alpha_1},f_2^{\alpha_2},\cdots, f_m^{\alpha_m}\big)(x)\big|\\
   &\leq&C_K\int_{(\mathbb{R}^{d})^m}\big|K(z,y_1,y_2,\cdots,y_m)-K(x,y_1,y_2,\cdots,y_m)\big|\prod\limits_{j=1}^m \,\big|f^{\alpha_j}_j(y_j)\big|\,d\mu_k(y_j)\\
    &\leq&C_K\int_{(\mathbb{R}^{d})^m} \big[\sum\limits_{n=1}^m \mu_k\big(B(z,d_G(z,y_n))\big)\big]^{-m} 
     \left[\frac{|z-x|}{\max\limits_{1\leq n\leq m} |z-y_n|}\right]^{\epsilon} \prod\limits_{j=1}^m \,\big|f^{\alpha_j}_j(y_j)\big|\,d\mu_k(y_j)\\
     &=&C_K\,\int\limits_{\big(\mathcal{O}(B^{**})\big)^l}\prod\limits_{s=1}^l\big|f^0_{j_s}(y_{j_s})\big|
      \int\limits_{\big(\mathbb{R}^d\setminus\mathcal{O}(B^{**})\big)^{m-l}}\big[\sum\limits_{n=1}^m \mu_k\big(B(z,d_G(z,y_n))\big)\big]^{-m} \\
      &&\times
     \left[\frac{2r(B)}{\max\limits_{1\leq n\leq m} |z-y_n|}\right]^{\epsilon} \prod\limits_{j\notin\,\{{j_1},{j_2},\cdots,{j_l}\}} \,\big|f_j(y_j)\big| 
   \,\prod\limits_{s=1}^ld\mu_k(y_{j_s})\,\prod\limits_{j\notin\,\{{j_1},{j_2},\cdots,{j_l}\}}d\mu_k(y_j)\\
 &\leq &CC_K\,\sum\limits_{n=1}^\infty\int\limits_{\big(\mathcal{O}(B^{**})\big)^l}\prod\limits_{s=1}^l\big|f^0_{j_s}(y_{j_s})\big|
      \int\limits_{\big(\mathcal{O}(3^{n}B^{**})\big)^{m-l}\setminus\big(\mathcal{O}(3^{n-1}B^{**})\big)^{m-l}}\!\!\!
      \big[\sum\limits_{s=1}^m \mu_k\big(B(z,d_G(z,y_s))\big)\big]^{-m} \\
      &&\times \left[\frac{r(B)}{\max\limits_{1\leq s\leq m} d_G(z,y_s)}\right]^{\epsilon} \prod\limits_{j\notin\,\{{j_1},{j_2},\cdots,{j_l}\}} \,\big|f_j(y_j)\big|\,\prod\limits_{s=1}^ld\mu_k(y_{j_s})\,\prod\limits_{j\notin\,\{{j_1},{j_2},\cdots,{j_l}\}}d\mu_k(y_j),\\
\end{eqnarray*}
where in the last step, we have used the fact that $ (\mathbb{R}^d\setminus\mathcal{O}(B^{**}))^{m-l}\subseteq\big(\mathbb{R}^d\big)^{m-l}\setminus\big(\mathcal{O}(B^{**})\big)^{m-l}$.
\par Using the inequalities $\max\limits_{1\leq s\leq m}d_G(z,y_s)\geq C\,3^nr(B)$ and $\sum\limits_{s=1}^m \mu_k\big(B(z,d_G(z,y_s))\big)\geq C\,\mu_k(3^n B)$, from above we write
\begin{eqnarray*}
     &&\big|\,\mathcal{T}\big(f_1^{\alpha_1},f_2^{\alpha_2},\cdots, f_m^{\alpha_m}\big)(z)-\mathcal{T}\big(f_1^{\alpha_1},f_2^{\alpha_2},\cdots, f_m^{\alpha_m}\big)(x)\big|\\
     &\leq&CC_K\,\sum\limits_{n=1}^\infty3^{-n\epsilon}\int\limits_{\big(\mathcal{O}(B^{**})\big)^l}\prod\limits_{s=1}^l\big|f^0_{j_s}(y_{j_s})\big|\,d\mu_k(y_{j_s})\\
     &&\times
      \int\limits_{\big(\mathcal{O}(3^{n}B^{**})\big)^{m-l}\setminus\big(\mathcal{O}(3^{n-1}B^{**})\big)^{m-l}}\big(\mu_k(3^n B)\big)^{-m}\prod\limits_{j\notin\,\{{j_1},{j_2},\cdots,{j_l}\}} \,\big|f_j(y_j)\big|\,d\mu_k(y_j)\\
      &\leq&CC_K\,\sum\limits_{n=1}^\infty 3^{-n\epsilon}\prod\limits_{j=1}^m\frac{1}{\mu_k(3^n B)}\int\limits_{\mathcal{O}(3^n B^{**})}\big|f_{j}(y_{j})\big|\,d\mu_k(y_{j})\\
       &\leq&CC_K\,\sum\limits_{n=1}^\infty 3^{-n\epsilon}\prod\limits_{j=1}^m\Big[\sum\limits_{\sigma\in G}\frac{1}{\mu_k(3^n B)}\int\limits_{3^n B^{**}}\big|f_{j}\circ\sigma(y_{j})\big|\,d\mu_k(y_{j})\Big]\\
       &=&CC_K\,\sum\limits_{n=1}^\infty 3^{-n\epsilon}\sum\limits_{\underset{\sigma _{n_s}\in G}{(n_1,n_2,\cdots,n_m)}}\Big[\prod\limits_{j=1}^m \frac{1}{\mu_k(3^n B^{**})}\int\limits_{3^n B^{**}}\big|f_{j}\circ\sigma_{n_j}(y_{j})\big|\,d\mu_k(y_{j})\Big]\\
       &\leq& C\, (C_K+A)\sum\limits_{\underset{\sigma _{n_s}\in G}{(n_1,n_2,\cdots,n_m)}} \mathcal{M}^k_{HL} \left(f_1\circ \sigma_{n_1}, f_2\circ \sigma_{n_2},\cdots,f_m\circ \sigma_{n_m}\right)(x).
\end{eqnarray*}
Now making use of (\ref{1st part domi by multi maxi}) and the last inequality in (\ref{2nd time diving T into two parts}), we conclude the proof of (\ref{esti of sharp maxi domi by multi maxi}).
\end{proof}
\begin{prop}\label{T f vectore lpw norm is domi by M f vector lpw norm}
Let $w$ be a weight in the class $A^k_{\infty}$ and $p\in[1/m,\infty)$. Then there exists a constant $C>0$ such that if $f_1,f_2,\cdots,f_m$ are bounded functions with compact support
\begin{enumerate}[label=(\roman*)]
    \item \label{part 1 of T f vectore lpw norm is domi by M f vector lpw norm}if $p>1/m$, then
    \begin{eqnarray*}
    &&\Big(\int_{\mathbb{R}^d}\big|\mathcal{T} \overrightarrow {f} (x)\big|^p w(x)\, d\mu_k(x)\Big)^{{1}/{p}}\\
    &\leq& C(C_K+A)\sum\limits_{\underset{\sigma _{n_s}\in G}{(n_1,n_2,\cdots,n_m)}}\Big(\int_{\mathbb{R}^d}\big(\mathcal{M}_{HL} ^k \overrightarrow {f_{\bf{\sigma}}} (x)\big)^p w(x)\, d\mu_k(x) \Big)^{{1}/{p}};
    \end{eqnarray*}
    \item\label{part 2 of T f vectore lpw norm is domi by M f vector lpw norm} if $p\geq 1/m$, then
    \begin{eqnarray*}
    &&\sup\limits_{t>0}\,t\Big(\int\limits_{ \{y\,\in{\mathbb{R}^d:\,|\mathcal{T} \overrightarrow {f} (y)|>t\}}}w(x)\,d\mu_k(x)\Big)^{{1}/{p}}\\
    &\leq& C(C_K+A) \sum\limits_{\underset{\sigma _{n_s}\in G}{(n_1,n_2,\cdots,n_m)}}\sup\limits_{t>0}\,t\Big(\int\limits_{ \{y\,\in{\mathbb{R}^d:\,\mathcal{M}^k_{HL} \overrightarrow {f_{\bf{\sigma}}} (y)>t\}}}w(x)\,d\mu_k(x)\Big)^{{1}/{p}},
    \end{eqnarray*}
\end{enumerate}
where we have used the notation $\overrightarrow {f_{\bf{\sigma}}}=\left(f_1\circ \sigma_{n_1}, \cdots,f_m\circ \sigma_{n_m}\right)$.
\end{prop}
\begin{proof}
We only prove \ref{part 1 of T f vectore lpw norm is domi by M f vector lpw norm} as \ref{part 2 of T f vectore lpw norm is domi by M f vector lpw norm} follows by similar arguments. For any $N\in \mathbb{N}$, define $w_N(x)=\min \{w(x), N\}$. Then the weight $w_N\in A^k_{\infty}$ (cf. \cite[p. 215]{Cruz-uribeBookWEATORDF}) and the constant $[w_N]_{A^k_\infty}$ in the $A^k_\infty$ condition, does not depend on $N$.
\par Also by Fatou's lemma 
$$\int_{\mathbb{R}^d}\big|\mathcal{T} \overrightarrow {f}(x)\big|^p\,w(x)\,d\mu_k(x)\leq \lim\limits_{N\to \infty}\int_{\mathbb{R}^d}\big|\mathcal{T} \overrightarrow {f}(x)\big|^p\,w_N(x)\,d\mu_k(x).$$
Since $f_1,f_2,\cdots,f_m$ are bounded functions with compact support, by our hypothesis, $\mathcal{T}\overrightarrow f\in L^{q,\,\infty}(\mathbb{R}^d,\, d\mu_k)$ which implies
for any $\nu\in(0,1/m)$, $\big|\mathcal{T}\overrightarrow f\big|^\nu$ is locally integrable on $\mathbb{R}^d$ (by similar method as in (\ref{1st part domi by multi maxi})).
 Next we claim that 
 \begin{equation}\label{m of tf to power nu is finite}
   \sup\limits_{t>0}t\,\Big(\int\limits_{\big\{y\in \mathbb{R}^d:\big({M}_{HL} ^k (|\mathcal{T}\overrightarrow {f}|^\nu)(y)\big)^{1/\nu}>t\big\}} w_N(x)\, d\mu_k(x) \Big)^{1/p_{0}}<\infty, 
 \end{equation}
 for some $0<p_0<p$.
 Taking $p_0=1/m$ and using the fact that $M^k_{HL}$ is bounded on $L^{r,\,\infty}(\mathbb{R}^d,d\mu_k)$ for $1\leq r<\infty$ (cf. \cite[p. 103]{GrafakosClassicalBook}), we have
\begin{eqnarray*}
 &&\sup\limits_{t>0}t\,\Big(\int\limits_{\big\{y\in \mathbb{R}^d:\big({M}_{HL} ^k (|\mathcal{T}\overrightarrow {f}|^\nu)(y)\big)^{1/\nu}>t\big\}} w_N(x)\, d\mu_k(x) \Big)^{m}\\
 &\leq& \|w_N\|_{L^\infty}^m\|{M}_{HL} ^k(|\mathcal{T}\overrightarrow {f}|^\nu)\|^{1/\nu}_{L^{1/m\nu,\, \infty}(d\mu_k)}\\
 &\leq& \|w_N\|_{L^\infty}^m \|{M}_{HL} ^k\|^{1/\nu}_{L^{1/m\nu,\, \infty}(d\mu_k)\rightarrow L^{1/m\nu,\, \infty}(d\mu_k)}\| |\mathcal{T}\overrightarrow {f}|^\nu\|^{1/\nu}_{L^{1/m\nu,\, \infty}(d\mu_k)}\\
 &=& \|w_N\|_{L^\infty}^m \|{M}_{HL} ^k\|^{1/\nu}_{L^{1/m\nu,\, \infty}(d\mu_k)\rightarrow L^{1/m\nu,\, \infty}(d\mu_k)}\| \mathcal{T}\overrightarrow {f}\|_{L^{1/m,\, \infty}(d\mu_k)}<\infty.
\end{eqnarray*}
This completes the proof of the claim (\ref{m of tf to power nu is finite}).
\par Therefore, by arguing as \cite[Lemma 4.11 (i)]{GrafkosMAMS} for the space of homogeneous type $(\mathbb{R}^d,| x-y |,d\mu_k) $ and using Proposition \ref{prop sharp maxi funct domi by multi maxi}, we finally write
\begin{eqnarray*}
&&\Big(\int_{\mathbb{R}^d}\big|\mathcal{T} \overrightarrow {f} (x)\big|^p w_N(x)\, d\mu_k(x)\Big)^{{1}/{p}} \\
&\leq &\Big(\int_{\mathbb{R}^d}\big({M}_{HL} ^k (|\mathcal{T}\overrightarrow {f}|^\nu)(x)\big)^{p/\nu}w_N(x)\, d\mu_k(x)\Big)^{{1}/{p}}\\
&\leq &C_{[w_N]_{A^k_\infty}}\Big(\int_{\mathbb{R}^d}\big({M}_{HL,\nu} ^{k,\#} (\mathcal{T}\overrightarrow {f})(x)\big)^{p}w_N(x)\, d\mu_k(x)\Big)^{{1}/{p}}\\
&\leq& C\,(C_K+A)\sum\limits_{\underset{\sigma _{n_s}\in G}{(n_1,n_2,\cdots,n_m)}} \Big(\int_{\mathbb{R}^d}\big(\mathcal{M}^k_{HL} \left(f_1\circ \sigma_{n_1}, \cdots,f_m\circ \sigma_{n_m}\right)(x)\big)^p w_N(x)\, d\mu_k(x)\Big)^{{1}/{p}}\\
&\leq& C\,(C_K+A)\sum\limits_{\underset{\sigma _{n_s}\in G}{(n_1,n_2,\cdots,n_m)}} \Big(\int_{\mathbb{R}^d}\big(\mathcal{M}^k_{HL} \left(f_1\circ \sigma_{n_1}, \cdots,f_m\circ \sigma_{n_m}\right)(x)\big)^p w(x)\, d\mu_k(x)\Big)^{{1}/{p}}.
\end{eqnarray*}
Recalling the fact that $[w_N]_{A^k_\infty}$ is independent of $N$,  we get the required result by taking $N\to\infty.$
\end{proof}
Now we are in a position to prove our main results regarding weighted inequalities for multilinear Dunkl-Calder\'on-Zygmund operators.
\begin{proof}[Proof of Theorem \ref{two weight calderon Zygmund} \ref{CZ two weigh weak thm}]
    Proof follows at once from Proposition \ref{T f vectore lpw norm is domi by M f vector lpw norm} \ref{part 2 of T f vectore lpw norm is domi by M f vector lpw norm}, Theorem \ref{multi maximal funct two weight} \ref{multi maximal two weigh weak thm} and the $G$-invariance of the weights.
\end{proof}
\begin{proof}[Proof of Theorem \ref{two weight calderon Zygmund} \ref{CZ two weight strong thm}] Similarly the proof follows form Proposition \ref{T f vectore lpw norm is domi by M f vector lpw norm} \ref{part 1 of T f vectore lpw norm is domi by M f vector lpw norm}, Theorem \ref{multi maximal funct two weight} \ref{multi maximal two weight strong thm} and the $G$-invariance of the weights.
\end{proof}
\begin{proof}[Proof of Theorem \ref{one weight calderon Zygmund} \ref{CZ one weigh weak thm}]
  Note that $\overrightarrow w\in A_{\overrightarrow P}^k$ implies that $\prod\limits_{j=1}^d w_j^{p/p_j}\in A^k_{\infty}$ (see \cite[Proposition 4.3]{GrafkosMAMS}). Hence, in this case also proof from Proposition \ref{T f vectore lpw norm is domi by M f vector lpw norm} \ref{part 2 of T f vectore lpw norm is domi by M f vector lpw norm}, Theorem \ref{multi maximal funct two weight} \ref{multi maximal two weigh weak thm} with $v=\prod\limits_{j=1}^d w_j^{p/p_j}$ and the $G$-invariance of the weights. 
\end{proof}
\begin{proof}[Proof of Theorem \ref{one weight calderon Zygmund} \ref{CZ one weigh strong thm}]
The proof can be completed by using Proposition \ref{T f vectore lpw norm is domi by M f vector lpw norm} \ref{part 1 of T f vectore lpw norm is domi by M f vector lpw norm}, Theorem \ref{multi maximal funct one weight}  and the $G$-invariance of the weights together with the property of the $A^k_{\overrightarrow P}$ weights as used in the last proof.
\end{proof}
\section{Proofs of The  Weighted Inequalities for Multipliers }\label{proof main theorem multiplier}

\par The first main result of this section is Coifman-Meyer \cite{CoifmenNHAOTAPIBL} type multiplier theorem in Dunkl setting.
\begin{thm}\label{bilinear multi main thm}
 Let $1<p,p_1,p_2<\infty$ with $1/p=1/p_1+1/p_2$ and $L\in \mathbb{N}$ be such that $L>2d+2\lfloor d_k \rfloor+4$. If ${\bf m}\in C^L\left(\mathbb{R}^d\times \mathbb{R}^d \setminus\{(0,0)\}\right)$ be a function satisfying
 \begin{equation*}
     |\partial^{\alpha}_{\xi}\partial^{\beta}_{\eta}{\bf m}(\xi,\eta)|\leq C_{\alpha,\,\beta}\left(|\xi|+|\eta|\right)^{-(|\alpha|+|\beta|)}
 \end{equation*}
 for all multi-indices $\alpha,\beta \in \left(\mathbb{N}\cup\{0\}\right)^d$ such that $|\alpha|+|\beta|\leq L$ and for all $(\xi,\eta)\in \mathbb{R}^d \times \mathbb{R}^d\setminus\{(0,0)\};$ then for all $f_1 \in L^{p_1}(\mathbb{R}^d,d\mu_k)$ and $ f_2 \in L^{p_2}(\mathbb{R}^d,d\mu_k)$ the following boundedness holds:
 $$  ||\mathcal{T}_{\bf m}(f_1,f_2)||_{L^{p}(d\mu_k)} \leq C\, ||f_1||_{L^{p_1}(d\mu_k)}\, ||f_2||_{L^{p_2}(d\mu_k)}.$$
\end{thm}
\begin{proof}
Let $\psi\in C^{\infty}(\mathbb{R}^d)$ be such that $supp\, \psi \subset \{\xi\in\mathbb{R}^d: 1/2\leq |\xi|\leq 2\}$ and 
$$\sum\limits_{j\in\mathbb{Z}}\psi_j(\xi)=1 \text{ for all } \xi\neq 0,$$
where $\psi_j(\xi)=\psi(\xi/2^j)$ for all $\xi\in\mathbb{R}^d.$
Then 
\begin{eqnarray*}
\mathcal{T}_{\bf m}(f_1,f_2)(x)&=&\int_{\mathbb{R}^{2d}}\sum\limits_{j_1\in\mathbb{Z}}\sum\limits_{j_2\in\mathbb{Z}} \psi_{j_1}(\xi)\psi_{j_2}(\eta) {\bf m}(\xi,\eta)\mathcal{F}_kf_1(\xi)\mathcal{F}_kf_2(\eta)E_k(ix,\xi)\\&& \times E_k(ix,\eta)\,d\mu_k(\xi)d\mu_k(\eta)\\
&=&\int_{\mathbb{R}^{2d}}\sum\limits_{|j_1-j_2|\leq 4}\cdots +\int_{\mathbb{R}^{2d}}\sum\limits_{j_1>j_2+4}\cdots +\int_{\mathbb{R}^{2d}}\sum\limits_{j_2>j_1+4}\cdots\\
&=:&\mathcal{T}_1(f_1,f_2)(x)+\mathcal{T}_2(f_1,f_2)(x)+\mathcal{T}_3(f_1,f_2)(x)
\end{eqnarray*}
We will calculate the estimates for $\mathcal{T}_1$, $\mathcal{T}_2$ and $\mathcal{T}_3$ separately.\\\\
\underline{Estimate of $\mathcal{T}_1$:}
\begin{eqnarray*}
\text{ For any } j\in\mathbb{Z} \text{ define } {\bf m}_j(\xi,\eta)&=&\psi_j(\xi)\sum\limits_{j_2:|j-j_2|\leq 4}\psi_{j_2}(\eta){\bf m}(\xi,\eta)\\
&=:&\psi_j(\xi)\phi_j(\eta){\bf m}(\xi,\eta),
\end{eqnarray*}
where $\phi(\eta)=\sum\limits_{|j|\leq 4}\psi_j(\eta)$ and $\phi_j(\eta)=\phi(\eta/2^j)$.
\par Then $supp\, \phi \subset \{\xi\in\mathbb{R}^d: 2^{-5}\leq |\xi|\leq 2^5\}$.
\par Let $\tilde\psi\in C^{\infty}(\mathbb{R}^d)$ be another function such that $0\leq \tilde\psi\leq 1$ and
$$\tilde \psi (\xi) =
\left\{
	\begin{array}{ll}
		0  & \mbox{if } |\xi|\notin [2^{-6},2^6],\\
		1 & \mbox{if } |\xi|\in [2^{-5},2^5].
	\end{array}
\right.$$
Then we have for any $j\in\mathbb{Z}$,
\begin{eqnarray*}
{\bf m}_j(\xi,\eta)&=& \psi_j(\xi)\phi_j(\eta){\bf m}(\xi,\eta)\\
&=&\tilde\psi_j(\xi)\psi_j(\xi)\tilde\psi_j(\eta)\phi_j(\eta){\bf m}(\xi,\eta)\\
&=&\tilde\psi_j(\xi)\tilde\psi_j(\eta){\bf m}_j(\xi,\eta).
\end{eqnarray*}
Now $supp\, {\bf m}_j \subset \big\{(\xi,\eta)\in\mathbb{R}^d \times \mathbb{R}^d:2^{j-6}\leq |\xi|\leq 2^{j+6},\ 2^{j-6}\leq |\eta|\leq 2^{j+6}\big\}$.
\par Define for all $j\in \mathbb{Z}$,
$$a_j(\xi,\eta)={\bf m}_j(2^j\xi,2^j\eta).$$
Then  $supp\, a_j \subset \big\{(\xi,\eta)\in\mathbb{R}^d \times \mathbb{R}^d:2^{-6}\leq |\xi|\leq 2^{6},\ 2^{-6}\leq |\eta|\leq 2^{6}\big\}$. Now using support of $a_j$ and smoothness assumption on ${\bf m}$, expanding $a_j$ in terms Fourier series over $[2^{-6},2^6]^d\times [2^{-6},2^6]^d$ we write
\begin{equation}\label{formula for Fourier series of aj}
a_j(\xi,\eta)=\sum\limits_{{\bf n}_1\in\mathbb{Z}^d}\sum\limits_{{\bf n}_2\in\mathbb{Z}^d}c_j({\bf n}_1,{\bf n}_2)\,e^{2\pi i (\langle \xi,\,{\bf n}_1\rangle+\langle \eta,\,{\bf n}_2\rangle )},
\end{equation}
where the Fourier coefficients $c_j({\bf n}_1,{\bf n}_2)$ are given by
$$c_j({\bf n}_1,{\bf n}_2)=\iint\limits_{[2^{-6},\,2^6]^d\times [2^{-6},\,2^6]^d}a_j(y,z)\,e^{-2\pi i (\langle y,\,{\bf n}_1\rangle+\langle z,\,{\bf n}_2\rangle )}\,dydz.$$
\par Thus we get
\begin{eqnarray*}
{\bf m}_j(\xi,\eta)&=&\sum\limits_{{\bf n}_1\in\mathbb{Z}^d}\sum\limits_{{\bf n}_2\in\mathbb{Z}^d}c_j({\bf n}_1,{\bf n}_2)\,e^{2\pi i (\langle \xi,\,{\bf n}_1\rangle+\langle \eta,\,{\bf n}_2\rangle )/2^j}\\
&=&\sum\limits_{{\bf n}_1\in\mathbb{Z}^d}\sum\limits_{{\bf n}_2\in\mathbb{Z}^d}c_j({\bf n}_1,{\bf n}_2)\,e^{2\pi i (\langle \xi,\,{\bf n}_1\rangle+\langle \eta,\,{\bf n}_2\rangle )/2^j}\tilde\psi_j(\xi)\tilde\psi_j(\eta).
\end{eqnarray*}
Substituting this in the expression for $\mathcal{T}_1$ and interchanging sum and integration, we obtain
\begin{eqnarray}\label{expression for T1}
&&\mathcal{T}_1(f_1,f_2)(x)\\
&=&\int_{\mathbb{R}^{2d}}\sum\limits_{j\in\mathbb{Z}} {\bf m}_j(\xi,\eta)\mathcal{F}_kf_1(\xi)\mathcal{F}_kf_2(\eta)E_k(ix,\xi)E_k(ix,\eta)\,d\mu_k(\xi)d\mu_k(\eta)\nonumber\\
&=&\sum\limits_{{\bf n}_1\in\mathbb{Z}^d}\sum\limits_{{\bf n}_2\in\mathbb{Z}^d}\sum\limits_{j\in\mathbb{Z}}c_j({\bf n}_1,{\bf n}_2) \left(\int_{\mathbb{R}^d}\tilde\psi _j(\xi)\,e^{2\pi i\langle {\bf n}_1,\, \xi\rangle/2^j}\mathcal{F}_kf_1(\xi)E_k(ix,\xi)\,d\mu_k(\xi)\right)\nonumber\\\nonumber
&&\times \left(\int_{\mathbb{R}^d}\tilde\psi _j(\eta)\,e^{2\pi i\langle {\bf n}_2,\, \eta\rangle/2^j}\mathcal{F}_kf_2(\eta)E_k(ix,\eta)\,d\mu_k(\eta)\right)\\
&=&\sum\limits_{{\bf n}_1\in\mathbb{Z}^d}\sum\limits_{{\bf n}_2\in\mathbb{Z}^d}\sum\limits_{j\in\mathbb{Z}}c_j({\bf n}_1,{\bf n}_2)\, \tilde\psi(2\pi {\bf n}_1,D_k/2^j)f_1(x)\,\tilde\psi(2\pi {\bf n}_2,D_k/2^j)f_2(x),\nonumber
\end{eqnarray}
where $\tilde\psi(2\pi {\bf n}_1,D/2^j)f_1(x)$ and $\tilde\psi(2\pi {\bf n}_2,D/2^j)f_2(x)$ are as in Section \ref{Littlewood-Plaey section}.
Now as $\psi$ and $\phi$ are supported compactly away from origin and ${\bf m}$ satisfies (\ref{usual deriv condtn m}), by Leibniz rule we have that for all $j\in\mathbb{Z}$,
\begin{equation}\label{m xi eta satisfies derivative condtn}
|\partial^{\alpha}_{\xi}\partial^{\beta}_{\eta}\left({\bf m}(2^j\xi,2^j\eta)\psi(\xi)\phi(\eta)\right)|\leq C_{\alpha,\beta}\left(|\xi|+|\eta|\right)^{-(|\alpha|+|\beta|)}  
\end{equation}
for all $\alpha,\beta \in \left(\mathbb{N}\cup\{0\}\right)^d$ such that $|\alpha|+|\beta|\leq L$ and for all $(\xi,\eta)\in \mathbb{R}^d \times \mathbb{R}^d\setminus\{(0,0)\}.$
\par Now for any $\alpha,\beta \in \left(\mathbb{N}\cup\{0\}\right)^d$ we have
\begin{eqnarray*}
   c_j({\bf n}_1,{\bf n}_2)&=&\iint\limits_{[2^{-6},\,2^6]^d\times [2^{-6},\,2^6]^d}{\bf m}(2^jy,2^jz)\psi(y)\phi(z)\,e^{-2\pi i (\langle y,\,{\bf n}_1\rangle+\langle z,\,{\bf n}_2\rangle )}\,dydz\\
   &=&\frac{C}{{\bf n}_1^\alpha {\bf n}_2^\beta}\  \iint\limits_{[2^{-6},\,2^6]^d\times [2^{-6},\,2^6]^d}{\bf m}(2^jy,2^jz)\psi(y)\phi(z)\,\partial^{\alpha}_{y}\partial^{\beta}_{z}e^{-2\pi i (\langle y,\,{\bf n}_1\rangle+\langle z,\,{\bf n}_2\rangle )}\,dydz.
\end{eqnarray*}
Hence applying integration by parts formula and using (\ref{m xi eta satisfies derivative condtn}), we get
\begin{equation}\label{decay of Fourier coeff}
   |c_j({\bf n}_1,{\bf n}_2)|\leq C_{L}\,\frac{1}{(1+|{\bf n}_1|+|{\bf n}_2|)^{L}}
\end{equation}
  Now using (\ref{decay of Fourier coeff})  and applying Cauchy-Schwarz inequality, from (\ref{expression for T1}) we have
\begin{eqnarray*}
|\mathcal{T}_1(f_1,f_2)(x)|&\leq& \sum\limits_{{\bf n}_1\in\mathbb{Z}^d}\sum\limits_{{\bf n}_2\in\mathbb{Z}^d}\sum\limits_{j\in\mathbb{Z}}\big|c_j({\bf n}_1,{\bf n}_2)\, \tilde\psi(2\pi {\bf n}_1,D_k/2^j)f_1(x)\,\tilde\psi(2\pi {\bf n}_2,D_k/2^j)f_2(x)\big|\\
&\leq&C\, \sum\limits_{{\bf n}_1\in\mathbb{Z}^d}\sum\limits_{{\bf n}_2\in\mathbb{Z}^d}\frac{C}{(1+|{\bf n}_1|+|{\bf n}_2|)^{L}}\Big(\sum\limits_{j\in \mathbb{Z}}|\tilde\psi(2\pi {\bf n}_1,D_k/2^j)f_1(x)|^2\Big)^{1/2}\\
&&\times \Big(\sum\limits_{j\in \mathbb{Z}}|\tilde\psi(2\pi {\bf n}_2,D_k/2^j)f_2(x)|^2\Big)^{1/2}
\end{eqnarray*}
Finally from H\"older's inequality and Theorem \ref{Littlewood-Paley l2 thm} and using the facts that $L>2d+2\lfloor d_k \rfloor+4$ and $n=\lfloor d_k \rfloor+2$, we obtain
\begin{eqnarray*}
  &&\Big(\int_{\mathbb{R}^d}|\mathcal{T}_1(f_1,f_2)(x)|^p\,d\mu_k(x)\Big)^{1/p}\\
  &\leq& C\, \sum\limits_{{\bf n}_1\in\mathbb{Z}^d}\sum\limits_{{\bf n}_2\in\mathbb{Z}^d}
\frac{1}{(1+|{\bf n}_1|+|{\bf n}_2|)^{L}} \,\Big\|\Big(\sum\limits_{j\in \mathbb{Z}}|\tilde\psi(2\pi {\bf n}_1,D_k/2^j)f_1|^2\Big)^{1/2}\Big\|_{L^{p_1}(d\mu_k)}\\
  &&\times \Big\|\Big(\sum\limits_{j\in \mathbb{Z}}|\tilde\psi(2\pi {\bf n}_2,D_k/2^j)f_2|^2\Big)^{1/2}\Big\|_{L^{p_2}(d\mu_k)}\\
  &\leq&C\, ||f_1||_{L^{p_1}(d\mu_k)}\, ||f_2||_{L^{p_2}(d\mu_k)}\sum\limits_{{\bf n}_1\in\mathbb{Z}^d}\sum\limits_{{\bf n}_2\in\mathbb{Z}^d} \frac{(|1+|{\bf n}_1|)^{n}(|1+|{\bf n}_2|)^{n}}{(1+|{\bf n}_1|+|{\bf n}_2|)^{L}}\\
&\leq&C\, ||f_1||_{L^{p_1}(d\mu_k)}\, ||f_2||_{L^{p_2}(d\mu_k)}
\end{eqnarray*}

 This concludes the proof for $\mathcal{T}_1.$ \\\\
\underline{Estimate of $\mathcal{T}_2$:}
\begin{eqnarray*}
\text{ For any } j\in\mathbb{Z}, \text{ define } {\bf m}_j(\xi,\eta)&=&\psi_j(\xi)\sum\limits_{j_2:j_2<j-4}\psi_{j_2}(\eta){\bf m}(\xi,\eta)\\
&=:&\psi_j(\xi)\phi_j(\eta){\bf m}(\xi,\eta),
\end{eqnarray*}
where $\phi(\eta)=\sum\limits_{j<- 4}\psi_j(\eta)$ and $\phi_j(\eta)=\phi(\eta/2^j)$.
\par Then $supp\, \phi \subset \{\xi\in\mathbb{R}^d:  |\xi|\leq 2^{-3}\}$.
\par Let $\tilde\psi,\ \tilde\phi\in C^{\infty}(\mathbb{R}^d)$ be another two functions such that $0\leq \tilde\psi,\ \tilde\phi\leq 1$ and
$$\tilde \psi (\xi) =
\left\{
	\begin{array}{ll}
		0  & \mbox{if } |\xi|\notin [\frac{2}{5},\frac{5}{2}],\\
		1 & \mbox{if } |\xi|\in [2^{-1},2];
	\end{array}
\right.$$
$$\tilde \phi (\xi) =
\left\{
	\begin{array}{ll}
		0  & \mbox{if } |\xi|\notin [0,2^{-2}],\\
		1 & \mbox{if } |\xi|\in [0,2^{-3}].
	\end{array}
\right.$$

Then we have for any $j\in\mathbb{Z}$,
\begin{eqnarray*}
{\bf m}_j(\xi,\eta)&=& \psi_j(\xi)\phi_j(\eta){\bf m}(\xi,\eta)\\
&=&\tilde\psi_j(\xi)\psi_j(\xi)\tilde\phi_j(\eta)\phi_j(\eta){\bf m}(\xi,\eta)\\
&=&\tilde\psi_j(\xi)\tilde\phi_j(\eta){\bf m}_j(\xi,\eta).
\end{eqnarray*}
Now $supp\, {\bf m}_j \subset \big\{(\xi,\eta)\in\mathbb{R}^d \times \mathbb{R}^d:2^{j-1}\leq |\xi|\leq 2^{j+1},\  |\eta|\leq 2^{j-3}\big\}$.
\par Define for all $j\in \mathbb{Z}$,
$$a_j(\xi,\eta)={\bf m}_j(2^j\xi,2^j\eta).$$
Then  $supp\, a_j \subset \big\{(\xi,\eta)\in\mathbb{R}^d \times \mathbb{R}^d:2^{-1}\leq |\xi|\leq 2^{},\ |\eta|\leq 2^{-3}\big\}$.
\par Although $\psi(\xi)\,\phi(\eta)$ does not vanish at $\eta=0$ but it is clear that it vanishes near $\xi=0$. Hence, we can repeat the arguments as in the case of $\mathcal{T}_1$ to obtain that
\begin{eqnarray}\label{estimate of T2}
 && \mathcal{T}_2(f_1,f_2)(x)\\
 &=& \sum\limits_{{\bf n}_1\in\mathbb{Z}^d}\sum\limits_{{\bf n}_2\in\mathbb{Z}^d}\sum\limits_{j\in\mathbb{Z}}c_j({\bf n}_1,{\bf n}_2)\, \tilde\psi(2\pi {\bf n}_1,D_k/2^j)f_1(x)\,\tilde\phi(2\pi {\bf n}_2,D_k/2^j)f_2(x)\nonumber  
 \end{eqnarray}
 and also (\ref{decay of Fourier coeff}) holds.
\par To complete the proof in this case, we need the following lemma. A version of this lemma can be found in \cite{WrobelABMVAFC}, however for the sake of correctness and completeness, we provide a proof here.
\begin{lem}\label{entering phi lemma}
  Let $\Phi$ be a smooth function on $\mathbb{R}^d$ such that $0\leq \Phi\leq 1$, $supp\, \Phi\subset\{\xi\in \mathbb{R}^d:2^{-6}\leq |\xi|\leq 2^6\}$ and $\Phi(\xi)=1$ for $2^{-5}\leq |\xi|\leq 2^5$. For $j\in \mathbb{Z}$, define $\Phi_j(\xi)=\Phi (\xi/2^j)$ and for $f\in \mathcal{S}(\mathbb{R}^d)$ define
$$\Phi(0,D_k/2^j)f(x)=\int_{\mathbb{R}^d}\Phi_j(\xi) \mathcal{F}_kf(\xi)E_k(ix,\xi)\,d\mu_k(\xi) .$$  
Then for any $j\in\mathbb{Z}$ and $x\in\mathbb{R}^d$,
\begin{eqnarray*}
&&\tilde\psi(2\pi {\bf n}_1,D_k/2^j)f_1(x)\,\tilde\phi(2\pi {\bf n}_2,D_k/2^j)f_2(x)\\
&=&\Phi(0,D_k/2^j)\left(\tilde\psi(2\pi {\bf n}_1,D_k/2^j)f_1(\cdot)\,\tilde\phi(2\pi {\bf n}_2,D_k/2^j)f_2(\cdot)\right)(x).
\end{eqnarray*}
\end{lem}
\begin{proof}
It is enough to prove that 
\begin{eqnarray*}
    &&\mathcal{F}_k\left(\tilde\psi(2\pi {\bf n}_1,D_k/2^j)f_1(\cdot)\,\tilde\phi(2\pi {\bf n}_2,D_k/2^j)f_2(\cdot)\right)(\xi)\\
    &=&\Phi_j(\xi)\mathcal{F}_k\left(\tilde\psi(2\pi {\bf n}_1,D_k/2^j)f_1(\cdot)\,\tilde\phi(2\pi {\bf n}_1,D_k/2^j)f_2(\cdot)\right)(\xi),
\end{eqnarray*}
for all $\xi\in\mathbb{R}^d.$
\par Applying properties of Dunkl convolution the above equality is equivalent to 
\begin{eqnarray*}
&& \left(\tilde\psi _j(\cdot)e^{\langle 2\pi i{\bf n}_1,\,\cdot\rangle/2^j}\mathcal{F}_kf_1(\cdot)\,*_k\,\tilde\phi _j(\cdot)e^{\langle 2\pi i{\bf n}_2,\,\cdot\rangle/2^j}\mathcal{F}_kf_2(\cdot)\right)(\xi)\\
&=&\Phi_j(\xi)\left(\tilde\psi _j(\cdot)e^{\langle 2\pi i{\bf n}_1,\,\cdot\rangle/2^j}\mathcal{F}_kf_1(\cdot)\,*_k\,\tilde\phi _j(\cdot)e^{\langle 2\pi i{\bf n}_2,\,\cdot\rangle/2^j}\mathcal{F}_kf_2(\cdot)\right)(\xi),
\end{eqnarray*}
for all $\xi\in\mathbb{R}^d.$
Again, to prove that the above equality holds, from definition of $\tilde\psi$, $\tilde\phi$ and $\Phi$, it suffices to show that
for any $f,\, g \in \mathcal{S}(\mathbb{R}^d)$ with $supp\, f\subset \big\{\xi \in \mathbb{R}^d:\frac{2^{j+1}}{5}\leq |\xi|\leq 5.2^{j-1}\big\}$, $supp\, g\subset \big\{\eta \in \mathbb{R}^d:  |\eta|\leq 2^{j-2}\big\}$, we have that 
$$supp\, (f*_kg)\subset \big\{\xi\in \mathbb{R}^d:2^{j-5}\leq |\xi|\leq 2^{j+5}\big\}.$$
 Take $x\in \mathbb{R}^d$ such that $|x|\notin [2^{j-5},\  2^{j+5}]$. Now 
 $$f*_kg(x)=\int\limits_{\frac{2^{j+1}}{5}\leq |y|\leq 5.2^{j-1}}f(y)\tau^k_xg(-y)\,d\mu_k(y)$$
 Now from \cite[Theorem 1.7]{HejnaHMT} or \cite[Theorem 5.1]{AmriTRIDA} we have 
 $$supp\, \tau^k_xg(-\cdot)\subset\big\{y\in \mathbb{R}^d:|2^{j-2}-|x|\,|\leq |y|\leq |x|+2^{j-2}\big\}.$$
 But $|x|\notin [2^{j-5},\  2^{j+5}]$ and $y\in supp\, \tau^k_xg(-\cdot)$ together implies either $|y|< \frac{2^{j+1}}{5}$ or $|y|>5.2^{j-1}$ and which implies $f*_kg(x)=0$. This completes the proof of the Lemma.
\end{proof}
Coming back to the proof for $\mathcal{T}_2$, by Lemma \ref{entering phi lemma}, from (\ref{estimate of T2}) we get
\begin{eqnarray*}
 && \mathcal{T}_2(f_1,f_2)(x)\\
 &=& \sum\limits_{{\bf n}_1\in\mathbb{Z}^d}\sum\limits_{{\bf n}_2\in\mathbb{Z}^d}\sum\limits_{j\in\mathbb{Z}}c_j({\bf n}_1,{\bf n}_2)\, \Phi(0,D_k/2^j)\left(\tilde\psi(2\pi {\bf n}_1,D_k/2^j)f_1(\cdot)\,\tilde\phi(2\pi {\bf n}_2,D_k/2^j)f_2(\cdot)\right)(x).
 \end{eqnarray*}
To prove $L^p$ boundedness of $\mathcal{T}_2$, take $g\in \mathcal{S}(\mathbb{R}^d)$  with $||g||_{L^{p'}(d\mu_k)}=1 $ where $1/p+1/p'=1$, then using {color{brown } Plancherel} formula for Dunkl transform, we have
\begin{eqnarray*}
&&\int_{\mathbb{R}^d}\mathcal{T}_2(f_1,f_2)(x)\,g(x)\,d\mu_k(x)\\ 
&=& \int_{\mathbb{R}^d}\sum\limits_{{\bf n}_1\in\mathbb{Z}^d}\sum\limits_{{\bf n}_2\in\mathbb{Z}^d}\sum\limits_{j\in\mathbb{Z}}c_j({\bf n}_1,{\bf n}_2)\,\\
&&\times \Phi(0,D_k/2^j)\left(\tilde\psi(2\pi {\bf n}_1,D_k/2^j)f_1(\cdot)\,\tilde\phi(2\pi {\bf n}_2,D_k/2^j)f_2(\cdot)\right)(x)\,
 g(x)\,d\mu_k(x)\\
&=&\sum\limits_{{\bf n}_1\in\mathbb{Z}^d}\sum\limits_{{\bf n}_2\in\mathbb{Z}^d}\sum\limits_{j\in\mathbb{Z}}c_j({\bf n}_1,{\bf n}_2)\,\\
&&\times\Big(\int_{\mathbb{R}^d}\Phi(0,D_k/2^j)\left(\tilde\psi(2\pi {\bf n}_1,D_k/2^j)f_1(\cdot)\,\tilde\phi(2\pi {\bf n}_2,D_k/2^j)f_2(\cdot)\right)(x) g(x)\,d\mu_k(x)\,\Big)\\
&=&\sum\limits_{{\bf n}_1\in\mathbb{Z}^d}\sum\limits_{{\bf n}_2\in\mathbb{Z}^d}\sum\limits_{j\in\mathbb{Z}}c_j({\bf n}_1,{\bf n}_2)\, \Big(\int_{\mathbb{R}^d}\tilde\psi(2\pi {\bf n}_1,D_k/2^j)f_1(x)\,\tilde\phi(2\pi {\bf n}_2,D_k/2^j)f_2(x)\\
&&\times \Phi(0,D_k/2^j)g(x)\,d\mu_k(x)\,\Big).
\end{eqnarray*}

Again, using Cauchy-Schwarz inequality, H\"older's inequality, decay condition (\ref{decay of Fourier coeff}), Theorem \ref{Littlewood-Paley l2 thm} and Theorem \ref{Littlewood-Plaey l infinty thm} and using the facts that $L>2d+2\lfloor d_k \rfloor+4$ and $n=\lfloor d_k \rfloor+2$, we have
\begin{eqnarray*}
&&\Big |\int_{\mathbb{R}^d}\mathcal{T}_2(f_1,f_2)(x)\,g(x)\,d\mu_k(x)\Big|\\  
&=& C\!\sum\limits_{{\bf n}_1\in\mathbb{Z}^d}\sum\limits_{{\bf n}_2\in\mathbb{Z}^d} \!
\frac{1}{(1+|{\bf n}_1|+|{\bf n}_2|)^{L}} \,\Big\|\!\Big(\!\sum\limits_{j\in \mathbb{Z}}|\tilde\psi(2\pi {\bf n}_1,D_k/2^j)f_1\, \tilde\phi(2\pi {\bf n}_2,D_k/2^j)f_2|^2\!\Big)^{1/2}\!\Big\|_{L^{p}(d\mu_k)}\\
&&\times \Big(\sum\limits_{j\in \mathbb{Z}}|\Phi(0,D_k/2^j)g|^2\Big)^{1/2}\Big\|_{L^{p'}(d\mu_k)}\\
&\leq& C\, ||g||_{L^{p'}(d\mu_k)}\sum\limits_{{\bf n}_1\in\mathbb{Z}^d}\sum\limits_{{\bf n}_2\in\mathbb{Z}^d} 
\frac{1}{(1+|{\bf n}_1|+|{\bf n}_2|)^{L}} \,\Big\|\Big(\sum\limits_{j\in \mathbb{Z}}|\tilde\psi(2\pi {\bf n}_1,D_k/2^j)f_1|^2\Big)^{1/2}\Big\|_{L^{p_1}(d\mu_k)}\\
  &&\times \Big\|\sup\limits_{j\in\mathbb{Z}}|\tilde\phi(2\pi {\bf n}_2,D_k/2^j)f_2\Big\|_{L^{p_2}(d\mu_k)}\\
  &\leq & ||f_1||_{L^{p_1}(d\mu_k)}\, ||f_2||_{L^{p_2}(d\mu_k)}\sum\limits_{{\bf n}_1\in\mathbb{Z}^d}\sum\limits_{{\bf n}_2\in\mathbb{Z}^d} 
\frac{(|1+|{\bf n}_1|)^{n}(|1+|{\bf n}_2|)^{n}}{(1+|{\bf n}_1|+|{\bf n}_2|)^{L}}\\
&\leq&C\, ||f_1||_{L^{p_1}(d\mu_k)}\, ||f_2||_{L^{p_2}(d\mu_k)}
\end{eqnarray*}
Hence the proof for $\mathcal{T}_2$ is completed.\\\\
\underline{Estimate of $\mathcal{T}_3$:}\\\\
The estimate for $\mathcal{T}_3$ follows exactly in the same as in the case of $\mathcal{T}_2$ hence it is omitted.
\end{proof}
Finally we will present the proofs of the weighted inequalities for bilinear Dunkl multipliers.  
\begin{proof}[Proofs of Theorem \ref{two weight bilinear multi} and Theorem \ref{one weight bilinear multi}]
 Let $\phi\in C^{\infty}(\mathbb{R}^{2d})$ be such that $supp\, \phi \subset \{(\xi,\eta)\in\mathbb{R}^d\times \mathbb{R}^d: 1/4\leq \big(|\xi|^2+|\eta|^2\big)^{1/2}\leq 4\}$ and 
$$\sum\limits_{j\in\mathbb{Z}}\phi(\xi/2^j,\eta/2^j)=1 \text{ for all } (\xi,\eta)\neq (0,0)$$
where we recall that the notation $|\cdot|$ stands for the usual norm on $\mathbb{R}^{d}$. Then
\begin{eqnarray*}
    {\bf m}(\xi,\eta)&=&\sum\limits_{j\in\mathbb{Z}} {\bf m}(\xi,\eta)\,\phi(\xi/2^j,\eta/2^j)\\
    &=:&\sum\limits_{j\in\mathbb{Z}} {\bf m}_j(\xi/2^j,\eta/2^j).
\end{eqnarray*}
\par In view of the multilinear Dunkl setting defined in Section \ref{statement main theorem}, for $x, y_1, y_2\in \mathbb{R}^d$, let us define
$$K_j(x,y_1,y_2)=\tau^{k^2}_{(x,x)}\mathcal{F}_{k^2}^{-1}\big({\bf m}(\cdot,\cdot)\,\phi(\cdot/2^j,\cdot/2^j)\big)\big((-y_1,-y_2)\big)$$
$$\text{ and }\widetilde{K}_j(x,y_1,y_2)=\tau^{k^2}_{(x,x)}\mathcal{F}_{k^2}^{-1}\, {\bf m}_j\big((-y_1,-y_2)\big).$$
Then $K_j(x,y_1,y_2)=2^{2jd_k}\widetilde{K}_j(2^jx,2^jy_1,2^jy_2)$ for all $x, y_1, y_2\in \mathbb{R}^d$ and for all $f_1,f_2\in \mathcal{S}(\mathbb{R}^d)$,
\begin{eqnarray*}
   \mathcal{T}_{\bf m}(f_1,f_2)(x)&=&\int_{\mathbb{R}^{2d}}{\bf m}(\xi,\eta)\mathcal{F}_kf_1(\xi)\mathcal{F}_kf_2(\eta)E_k(ix,\xi)E_k(ix,\eta)\,d\mu_k(\xi)d\mu_k(\eta)\\
   &=&\sum\limits_{j\in\mathbb{Z}}\int_{\mathbb{R}^{2d}}{\bf m}_j(\xi/2^j,\eta/2^j)\mathcal{F}_{k^2}\big(f_1\otimes f_2\big)\big((\xi,\eta)\big)\, E_{k^2}\big(i(x,x), (\xi,\eta)\big)\\
   &&\times d\mu_{k^2}\big((\xi,\eta)\big)\\
    &=&\sum\limits_{j\in\mathbb{Z}}\int_{\mathbb{R}^{2d}}K_j(x,y_1,y_2)\,f_1(y_1)f_2(y_2)\,d\mu_k(y_1)d\mu_k(y_2)\\
    &=:&\int_{\mathbb{R}^{2d}}K(x,y_1,y_2)\,f_1(y_1)f_2(y_2)\,d\mu_k(y_1)d\mu_k(y_2).
\end{eqnarray*}
Having Theorem \ref{bilinear multi main thm} already proved, proofs of Theorem \ref{two weight bilinear multi} and Theorem \ref{one weight bilinear multi} will follow directly from Theorem \ref{two weight calderon Zygmund} and Theorem \ref{one weight calderon Zygmund},
if we can show that integral kernel $K$ of $\mathcal{T}_{\bf m}$, satisfies the  size estimate (\ref{size estimate of kernel}) and smoothness estimates (\ref{smoothness estimate of kernel}) for $m=2.$ For that, we need to show that for any $x, x', y_1, y_2\in \mathbb{R}^d$,
\begin{eqnarray}\label{multiplier size estimate last}
&&|K(x,y_1,y_2)|\\
&\leq& C\,\Big[\mu_k\big(B(x, d_G(x,y_1))\big)+\mu_k\big(B(x, d_G(x,y_2))\big)\Big]^{-2}\frac{d_G(x,y_1)+d_G(x,y_2)}{|x-y_1|+|x-y_2|}\nonumber
\end{eqnarray}
 for $d_G(x,y_1)+ d_G(x,y_2)>0;$
\begin{eqnarray}\label{multiplier smtness estimate last y2 changing}
    && |K(x,y_1,y_2)-K(x,y_1,y'_2)|\\
&\leq& C\,\Big[\mu_k\big(B(x, d_G(x,y_1))\big)+\mu_k\big(B(x, d_G(x,y_2))\big)\Big]^{-2} \frac{|y_2-y'_2|}{\max\{|x-y_1|,\,|x-y_2|\}}\nonumber
\end{eqnarray}
for $|y_2-y'_2|<\max\{d_G(x,y_1)/2,\, d_G(x,y_2)/2\}$;
\begin{eqnarray}\label{multiplier smtness estimate last y1 changing}
    && |K(x,y_1,y_2)-K(x,y'_1,y_2)|\\
&\leq& C\,\Big[\mu_k\big(B(x, d_G(x,y_1))\big)+\mu_k\big(B(x, d_G(x,y_2))\big)\Big]^{-2} \frac{|y_1-y'_1|}{\max\{|x-y_1|,\,|x-y_2|\}}\nonumber
\end{eqnarray}
for $|y_1-y'_1|<\max\{d_G(x,y_1)/2,\, d_G(x,y_2)/2\}$\\
and 
\begin{eqnarray}\label{multiplier smtness estimate last x changing}
 && |K(x,y_1,y_2)-K(x',y_1,y_2)|\\
&\leq& C\,\Big[\mu_k\big(B(x, d_G(x,y_1))\big)+\mu_k\big(B(x, d_G(x,y_2))\big)\Big]^{-2} \frac{|x-x'|}{\max\{|x-y_1|,\,|x-y_2|\}}\nonumber
\end{eqnarray}
for $|x-x'|<\max\{d_G(x,y_1)/2,\, d_G(x,y_2)/2\}$.\\\\
\underline{Proof of the inequality (\ref{multiplier size estimate last})}\\\\
The condition (\ref{usual deriv condtn m}) assures that 
\begin{equation}\label{Lth smooth condtn for mj}
\sup\limits_{j\in \mathbb{Z}}\|{\bf m}_j\|_{C^L(\mathbb{R}^{2d})}\leq C.
\end{equation}
Since 
$$\widetilde{K}_j(x,y_1,y_2)=\int_{\mathbb{R}^{2d}}{\bf m}_j(\xi,\eta)\, E_{k^2}\big(i(\xi,\eta),\,(x,x)\big)E_{k^2}\big(-i(\xi,\eta),\,(y_1,y_2)\big)\,d\mu_{k^2}\big((\xi,\eta)\big),$$
by applying \cite[eq.(4.30)]{HejnaRODTONRK} for $\mathbb{R}^{2d}$, we write 
\begin{eqnarray*}
    |\widetilde{K}_j(x,y_1,y_2)|
    &\leq& \frac{C}{\big[\mu_{k^2}\big(B((x,x), 1)\big) \mu_{k^2}\big(B((y_1,y_2), 1)\big)\big]^{1/2}}\\
    &&\times \frac{1}{1+\big(|x-y_1|^2+|x-y_2|^2\big)^{1/2}}\, \frac{1}{\big[1+d_{G\times G}\big((x,x),\, (y_1,y_2)\big)\big]^{L-1}}.
\end{eqnarray*}
Therefore, using (\ref{VOLRADREL}) for $\mathbb{R}^{2d}$ we get
\begin{eqnarray}\label{dividing kernel of multilplier in two sums}
&&|K(x, y_1, y_2)| \\
&\leq& \sum\limits_{j\in \mathbb{Z}}|K_j(x,y_1,y_2)|\nonumber\\
&=&\sum\limits_{j\in \mathbb{Z}}2^{2jd_k}\,|\widetilde{K}_j(2^jx,2^jy_1,2^jy_2)|\nonumber\\
&\leq& \sum\limits_{j\in \mathbb{Z}}\frac{C\,2^{2jd_k}}{\big[\mu_{k^2}\big(B((2^jx,2^jx), 1)\big) \mu_{k^2}\big(B((2^jy_1,2^jy_2), 1)\big)\big]^{1/2}}\nonumber\\
    &&\times \frac{1}{1+2^j\,\big(|x-y_1|^2+|x-y_2|^2\big)^{1/2}}\, \frac{1}{\big[1+2^j\,d_{G\times G}\big((x,x),\, (y_1,y_2)\big)\big]^{L-1}}\nonumber\\
    &\leq& C\,\sum\limits_{j\in \mathbb{Z}}\frac{1}{\big[\mu_{k^2}\big(B((x,x), 2^{-j})\big) \mu_{k^2}\big(B((y_1,y_2), 2^{-j})\big)\big]^{1/2}}\nonumber\\
    &&\times \frac{1}{1+2^j\,\big(|x-y_1|^2+|x-y_2|^2\big)^{1/2}}\, \frac{1}{\big[1+2^j\,d_{G\times G}\big((x,x),\, (y_1,y_2)\big)\big]^{L-1}}\nonumber\\
    &=&C \sum\limits_{j\in \mathbb{Z}: \,2^j\,d_{G\times G}((x,x),\, (y_1,y_2))\leq 1} \cdots +\sum\limits_{j\in \mathbb{Z}:\,2^j\,d_{G\times G}((x,x),\, (y_1,y_2))>1}\cdots .\nonumber
\end{eqnarray}
Again, from the discussion in Section \ref{statement main theorem}, using the product nature of the root system, it is not too hard to see that
\begin{align}\label{comparison btween linear and bilinear objects}
\begin{split}
\mu_{k^2}\big(B((z_1,z_2), r)\big)\sim
\mu_k\big(B(z_1,r)\big)\,\mu_k\big(B(z_2,r)\big),\\
d_{G\times G}\big((z,z),\, (z_1,z_2)\big)\sim d_G(z,z_1)+d_G(z,z_2)\\
\text{ and }\mu_{k}\big(B(z, r_1+r_2)\big)\geq C\,[\mu_{k}\big(B(z, r_1)\big)+\mu_{k}\big(B(z, r_2)\big)]
\end{split}
\end{align}
for all $z, z_1, z_2\in \mathbb{R}^d$ and $r, r_1, r_2>0.$
\par Now, if $2^j\,d_{G\times G}((x,x),\, (y_1,y_2))\leq 1$, by applying (\ref{VOLRADREL}) for $\mathbb{R}^{2d}$ and the relations (\ref{comparison btween linear and bilinear objects}), we deduce 
\begin{eqnarray*}
    \frac{1}{\mu_{k^2}\big(B((x, x), 2^{-j})\big)} &\leq& C\,\frac{[2^j\,d_{G\times G}((x,x),\, (y_1,y_2))]^{2d} }{\mu_{k^2}\big(B((x,x), d_{G\times G}((x,x),\, (y_1,y_2)))\big) }\\
    &\leq& C\,\frac{[2^j\,d_{G\times G}((x,x),\, (y_1,y_2))]^{2d} }{[\mu_k(B(x,d_G(x,y_1)))+\mu_k(B(x,d_G(x,y_2)))]^2}
\end{eqnarray*}
and 
\begin{eqnarray*}
    \frac{1}{\mu_{k^2}\big(B((y_1, y_2), 2^{-j})\big)}&\leq& C\,\frac{[2^j\,d_{G\times G}((x,x),\, (y_1,y_2))]^{2d} }{\mu_{k^2}\big(B((y_1,y_2), d_{G\times G}((x,x),\, (y_1,y_2)))\big) }\\
     &\leq&C\,\frac{[2^j\,d_{G\times G}((x,x),\, (y_1,y_2))]^{2d} }{\mu_{k^2}\big(B((x,x), d_{G\times G}((x,x),\, (y_1,y_2)))\big) }\\
    &\leq& C\,\frac{[2^j\,d_{G\times G}((x,x),\, (y_1,y_2))]^{2d} }{[\mu_k(B(x,d_G(x,y_1)))+\mu_k(B(x,d_G(x,y_2)))]^2},
\end{eqnarray*}
where in the second last inequality we used (\ref{V(x,y,d(x,y) comparison}) for $\mathbb{R}^{2d}$ .
\par Using the above two inequalities in (\ref{dividing kernel of multilplier in two sums}), we have
\begin{eqnarray}\label{estimate of the sum d less than}
    &&\sum\limits_{j\in \mathbb{Z}: \,2^j\,d_{G\times G}((x,x),\, (y_1,y_2))\leq 1}\frac{1}{\big[\mu_{k^2}\big(B((x,x), 2^{-j})\big) \mu_{k^2}\big(B((y_1,y_2), 2^{-j})\big)\big]^{1/2}}\\
    &&\times \frac{1}{1+2^j\,\big(|x-y_1|^2+|x-y_2|^2\big)^{1/2}}\, \frac{1}{\big[1+2^j\,d_{G\times G}\big((x,x),\, (y_1,y_2)\big)\big]^{L-1}}\nonumber\\
    &\leq&  C\,\sum\limits_{j\in \mathbb{Z}: \,2^j\,d_{G\times G}((x,x),\, (y_1,y_2))\leq 1}\frac{[2^j\,d_{G\times G}((x,x),\, (y_1,y_2))]^{2d} }{[\mu_k(B(x,d_G(x,y_1)))+\mu_k(B(x,d_G(x,y_2)))]^2}\nonumber\\
    &&\times \frac{1}{2^j(|x-y_1|+|x-y_2|)}\nonumber\\
    &\leq&  C\,\Big[\mu_k\big(B(x, d_G(x,y_1)\big)+\mu_k\big(B(x, d_G(x,y_2)\big)\Big]^{-2}\frac{d_G(x,y_1)+d_G(x,y_2)}{|x-y_1|+|x-y_2|}\nonumber\\
    &&\times \sum\limits_{j\in \mathbb{Z}: \,2^j\,d_{G\times G}((x,x),\, (y_1,y_2))\leq 1}[2^j\,d_{G\times G}((x,x),\, (y_1,y_2))]^{2d-1} \nonumber\\
    &\leq&  C\,\Big[\mu_k\big(B(x, d_G(x,y_1)\big)+\mu_k\big(B(x, d_G(x,y_2)\big)\Big]^{-2}\frac{d_G(x,y_1)+d_G(x,y_2)}{|x-y_1|+|x-y_2|}.\nonumber
\end{eqnarray}
Similarly, for the second sum in (\ref{dividing kernel of multilplier in two sums}), where $2^j\,d_{G\times G}((x,x),\, (y_1,y_2))> 1$, we get
\begin{eqnarray}\label{estimate of the sum d>>}
   &&\sum\limits_{j\in \mathbb{Z}: \,2^j\,d_{G\times G}((x,x),\, (y_1,y_2))> 1}\frac{1}{\big[\mu_{k^2}\big(B((x,x), 2^{-j})\big) \mu_{k^2}\big(B((y_1,y_2), 2^{-j})\big)\big]^{1/2}}\\
    &&\times \frac{1}{1+2^j\,\big(|x-y_1|^2+|x-y_2|^2\big)^{1/2}}\, \frac{1}{\big[1+2^j\,d_{G\times G}\big((x,x),\, (y_1,y_2)\big)\big]^{L-1}}\nonumber\\
    &\leq&  C\,\sum\limits_{j\in \mathbb{Z}: \,2^j\,d_{G\times G}((x,x),\, (y_1,y_2))> 1}\frac{[2^j\,d_{G\times G}((x,x),\, (y_1,y_2))]^{2d_k} }{[\mu_k(B(x,d_G(x,y_1)))+\mu_k(B(x,d_G(x,y_2)))]^2}\nonumber\\
    &&\times \frac{1}{2^j(|x-y_1|+|x-y_2|)}\frac{1}{\big[2^j\,d_{G\times G}\big((x,x),\, (y_1,y_2)\big)\big]^{L-1}}\nonumber\\
    &\leq&  C\,\Big[\mu_k\big(B(x, d_G(x,y_1)\big)+\mu_k\big(B(x, d_G(x,y_2)\big)\Big]^{-2}\frac{d_G(x,y_1)+d_G(x,y_2)}{|x-y_1|+|x-y_2|}\nonumber\\
    &&\times \sum\limits_{j\in \mathbb{Z}: \,2^j\,d_{G\times G}((x,x),\, (y_1,y_2))> 1}\frac{1}{[2^j\,d_{G\times G}((x,x),\, (y_1,y_2))]^{L-2d_k}} \nonumber\\
    &\leq&  C\,\Big[\mu_k\big(B(x, d_G(x,y_1)\big)+\mu_k\big(B(x, d_G(x,y_2)\big)\Big]^{-2}\frac{d_G(x,y_1)+d_G(x,y_2)}{|x-y_1|+|x-y_2|}.\nonumber
\end{eqnarray}
Substituting (\ref{estimate of the sum d less than}) and (\ref{estimate of the sum d>>}) in (\ref{dividing kernel of multilplier in two sums}), we complete the proof of (\ref{multiplier size estimate last}).\\\\
\underline{Proof of the inequality (\ref{multiplier smtness estimate last y2 changing})}\\\\
It is easy to see that 
$$d_G(x, y_1)\leq d_{G\times G}((x,x),\, (y_1,y_2)) \text{ and }d_G(x, y_2)\leq d_{G\times G}((x,x),\, (y_1,y_2)).$$
So the condition $|y_2-y'_2|<\max\{d_G(x,y_1)/2,\, d_G(x,y_2)/2\}$ implies that the $2d$-norm \footnote{ We have used the same notation $|\cdot|$ for norms on $\mathbb{R}^d$ and $\mathbb{R}^{2d}$.}
\begin{eqnarray*}
|(y_1, y_2)-(y_1, y'_2)|&\leq& d_{G\times G}((x,x),\, (y_1,y_2))/2,
\end{eqnarray*}
which further implies 
$$d_{G\times G}\big((x, x), (y_1, y_2)\big)\sim d_{G\times G}\big((x, x), (y_1, y'_2)\big),\ |(x, x)-(y_1, y_2)|\sim |(x, x)-(y_1, y'_2)| \text{ and }$$ 
$$V_{G\times G}\big((x, x), (y_1, y_2), d_{G\times G}\big((x, x), (y_1, y_2)\big)\big)\sim V_{G\times G}\big((x, x), (y_1, y'_2), d_{G\times G}\big((x, x), (y_1, y'_2)\big)\big).$$
By applying the techniques used in the proof of Theorem \ref{Littlewood-Paley l2 thm} in $\mathbb{R}^{2d}$ and using (\ref{Lth smooth condtn for mj}) in place of (\ref{growth of u}), we get
\begin{eqnarray}\label{bilimear before using symmetry}
 && 2^{2jd_k}|\widetilde{K}_j(2^jx,2^jy_1,2^jy_2)-\widetilde{K}_j(2^jx,2^jy_1,2^jy'_2)|\\
     &\leq& C\,  \frac{|(y_1, y_2)-(y_1, y'_2)|}{\big|(x, x)- (y_1, y_2)\big|}\,\frac{\big(2^jd_{G\times G}\big((x, x), (y_1, y_2)\big)\big)^{2d} +\big(2^jd_{G\times G}\big((x, x), (y_1, y_2)\big)\big)^{2d_k}}{V_{G\times G}\big((x, x), (y_1, y_2), d_{G\times G}\big((x, x), (y_1, y_2)\big)\big)}\nonumber\\
     &&\times \frac{1}{\left(1+2^jd_{G\times G}\big((x, x), (y_1, y_2)\right)^{L-1}}.\nonumber
\end{eqnarray}
Hence, using (\ref{V(x,y,d(x,y) comparison}) for $\mathbb{R}^{2d}$, and (\ref{comparison btween linear and bilinear objects}) repeatedly, from (\ref{bilimear before using symmetry}) we have
\begin{eqnarray*}
&& |K(x,y_1,y_2)-K(x,y_1,y'_2)|\\
&\leq& \sum\limits_{j\in \mathbb{Z}}2^{2jd_k}|\widetilde{K}_j(2^jx,2^jy_1,2^jy_2)-\widetilde{K}_j(2^jx,2^jy_1,2^jy'_2)|\\
 &\leq& C\,  \frac{|(y_1, y_2)-(y_1, y'_2)|}{\big|(x, x)- (y_1, y_2)\big|}\,\frac{1}{V_{G\times G}\big((x, x), (y_1, y_2), d_{G\times G}\big((x, x), (y_1, y_2)\big)\big)}\nonumber\\
     &&\times\sum\limits_{j\in \mathbb{Z}} \frac{\big(2^jd_{G\times G}\big((x, x), (y_1, y_2)\big)\big)^{2d} +\big(2^jd_{G\times G}\big((x, x), (y_1, y_2)\big)\big)^{2d_k}}{\left(1+2^jd_{G\times G}\big((x, x), (y_1, y_2)\right)^{L-1}}\\
     &\leq& C\,  \frac{|y_2- y'_2|}{|x-y_1|+|x-y_2|}\,\frac{1}{\mu_{k^2}\big(B((x, x), d_{G\times G}((x, x), (y_1, y_2)))\big)}\\
     &&\times\sum\limits_{j\in \mathbb{Z}} \frac{\big(2^jd_{G\times G}\big((x, x), (y_1, y_2)\big)\big)^{2d} +\big(2^jd_{G\times G}\big((x, x), (y_1, y_2)\big)\big)^{2d_k}}{\left(1+2^jd_{G\times G}\big((x, x), (y_1, y_2)\right)^{L-1}}\\
     &\leq& C\,  \frac{|y_2- y'_2|}{|x-y_1|+|x-y_2|}\,\frac{1}{\big[\mu_k\big(B(x, d_{G\times G}((x, x), (y_1, y_2)))\big)\big]^2}\\
     &&\times\sum\limits_{j\in \mathbb{Z}} \frac{\big(2^jd_{G\times G}\big((x, x), (y_1, y_2)\big)\big)^{2d} +\big(2^jd_{G\times G}\big((x, x), (y_1, y_2)\big)\big)^{2d_k}}{\left(1+2^jd_{G\times G}\big((x, x), (y_1, y_2)\right)^{L-1}}\\
     &\leq& C\,  \frac{|y_2- y'_2|}{\max\{|x-y_1|,\,|x-y_2|\}}\,\frac{1}{\big[\mu_k\big(B(x, d_{G}(x, y_1))\big)+\mu_k\big(B(x, d_{G}(x, y_2))\big) \big]^2}\\
     &&\times\sum\limits_{j\in \mathbb{Z}} \frac{\big(2^jd_{G\times G}\big((x, x), (y_1, y_2)\big)\big)^{2d} +\big(2^jd_{G\times G}\big((x, x), (y_1, y_2)\big)\big)^{2d_k}}{\left(1+2^jd_{G\times G}\big((x, x), (y_1, y_2)\right)^{L-1}}\\
      &\leq& C\,{\big[\mu_k\big(B(x, d_{G}(x, y_1))\big)+\mu_k\big(B(x, d_{G}(x, y_2))\big) \big]^{-2}}\frac{|y_2- y'_2|}{\max\{|x-y_1|,\,|x-y_2|\}},
\end{eqnarray*}
where the convergence of the last sum can be shown in exact same way as in the proof of Theorem \ref{Littlewood-Paley l2 thm}.\\\\
\underline{Proof of the inequality (\ref{multiplier smtness estimate last y1 changing})}\\\\
The proof of (\ref{multiplier smtness estimate last y1 changing}) is exactly the same as the proof of (\ref{multiplier smtness estimate last y2 changing}) with interchange of the roles of $y_1$ and $y_2$.\\\\
\underline{Proof of the inequality (\ref{multiplier smtness estimate last x changing})}\\\\
 Note that $\widetilde{K}_j$ can be written as 
 $$\widetilde{K}_j(x,y_1,y_2)=\int_{\mathbb{R}^{2d}}{\bf m}_j(\xi,\eta)\, E_{k^2}\big(i(\xi,\eta), (-y_1,-y_2)\big)E_{k^2}\big(-i(\xi,\eta), (-x,-x)\big)\,d\mu_{k^2}\big((\xi,\eta)\big),$$
Now, the condition $|x-x'|<\max\{d_G(x,y_1)/2,\, d_G(x,y_2)/2\}$ implies that the $2d$-norm
\begin{eqnarray*}
|(x, x)-(x', x')|&\leq& \sqrt{2}\, d_{G\times G}((x,x),\, (y_1,y_2))/2,
\end{eqnarray*}
 which further implies
$$d_{G\times G}\big((x, x), (y_1, y_2)\big)\sim d_{G\times G}\big((x', x'), (y_1, y_2)\big),\ |(x, x)-(y_1, y_2)|\sim |(x', x')-(y_1, y_2)| \text{ and }$$ 
$$V_{G\times G}\big((x, x), (y_1, y_2), d_{G\times G}\big((x, x), (y_1, y_2)\big)\big)\sim V_{G\times G}\big((x', x'), (y_1, y_2), d_{G\times G}\big((x', x'), (y_1, y_2)\big)\big).$$
Thus in this case also, rest of the proof can be carried forward in the same way as in the proof of (\ref{multiplier smtness estimate last y2 changing}).
\end{proof}

\bibliographystyle{abbrv}

%  This inserts the bib file, biblio.bib
\bibliography{bibliog}

% This command signals the end of the file.
\end{document}